\documentclass{amsart}
\usepackage[hmargin=30mm, vmargin=30mm]{geometry}
\usepackage{hyperref}
\usepackage[english]{babel}
\usepackage{amsmath,amssymb}
\usepackage{mathrsfs}
\usepackage{dsfont}
\usepackage{amscd}
\usepackage{lmodern}
\numberwithin{equation}{section}
\usepackage{mathtools}

\newtheorem{thmintro}{Theorem}

\newtheorem{theorem}{Theorem}[section]

\newtheorem{lemma}[theorem]{Lemma}
\newtheorem{prop}[theorem]{Proposition}
\theoremstyle{definition}

\newtheorem{remark}[theorem]{Remark}
\newtheorem{example}[theorem]{Example}

\newcommand{\CC}{\mathbb{C}}
\newcommand{\HHH}{\mathbb{H}}
\newcommand{\KK}{\mathbb{K}}
\newcommand{\NN}{\mathbb{N}}
\newcommand{\QQ}{\mathbb{Q}}
\newcommand{\RR}{\mathbb{R}}
\newcommand{\ZZ}{\mathbb{Z}}
\newcommand{\BBB}{\mathcal{B}}
\newcommand{\III}{\mathcal{I}}
\newcommand{\LLL}{\mathcal{L}}
\newcommand{\TTT}{\mathcal{T}}
\newcommand{\WW}{\mathcal{W}}
\newcommand{\HH}{\mathcal{H}}
\newcommand{\PP}{\mathcal{P}}
\newcommand{\g}{\mathfrak{g}}
\newcommand{\kk}{\mathfrak{k}}
\newcommand{\uu}{\mathfrak{u}}
\newcommand{\ttt}{\mathfrak{t}}
\newcommand{\bc}{\mathbf{c}}
\newcommand{\bd}{\mathbf{d}}

\newcommand{\inv}{^{-1}}
\newcommand{\la}{\lambda}
\newcommand{\co}{\colon\thinspace}
\newcommand{\gl}{{\mathfrak g}{\mathfrak l}}

\DeclareMathOperator{\ad}{ad}
\DeclareMathOperator{\Ad}{Ad}
\DeclareMathOperator{\sppan}{span}
\DeclareMathOperator{\GL}{GL}
\DeclareMathOperator{\End}{End}
\DeclareMathOperator{\Aut}{Aut}
\DeclareMathOperator{\id}{id}
\DeclareMathOperator{\im}{im}
\DeclareMathOperator{\der}{der}
\DeclareMathOperator{\diag}{diag}
\DeclareMathOperator{\conv}{conv}

\begin{document}

\renewcommand{\proofname}{{\bf Proof}}

\title{Isomorphisms of twisted Hilbert loop algebras}

\author[Timoth\'ee Marquis]{Timoth\'ee \textsc{Marquis}$^*$}
\address{Department Mathematik, FAU Erlangen-Nuernberg, Cauerstrasse 11, 91058 Erlangen, Germany}
\email{marquis@math.fau.de}
\thanks{$^*$Supported by a Marie Curie Intra-European Fellowship}

\author[Karl-Hermann Neeb]{Karl-Hermann \textsc{Neeb}$^\dagger$}
\address{Department Mathematik, FAU Erlangen-Nuernberg, Cauerstrasse 11, 91058 Erlangen, Germany}
\email{neeb@math.fau.de}
\thanks{$^\dagger$Supported by DFG-grant NE 413/7-2, Schwerpunktprogramm ``Darstellungstheorie"}

\begin{abstract}{
The closest infinite dimensional relatives of compact Lie algebras 
are Hilbert--Lie algebras, i.e. real Hilbert spaces with a Lie algebra 
structure for which the scalar product is invariant. 
Locally affine Lie algebras (LALAs) 
correspond to double extensions of (twisted) loop algebras 
over simple Hilbert--Lie algebras $\mathfrak{k}$, also called affinisations of $\mathfrak{k}$. 
They possess a root space decomposition 
whose corresponding root system is a locally affine root system 
of one of the $7$ families $A_J^{(1)}$, $B_J^{(1)}$, $C_J^{(1)}$, $D_J^{(1)}$, $B_J^{(2)}$, $C_J^{(2)}$ and $BC_J^{(2)}$ for some infinite set $J$. To each of these types corresponds a ``minimal" affinisation of some simple Hilbert--Lie algebra $\mathfrak{k}$, which we call standard.

In this paper, we give for each affinisation $\mathfrak{g}$ of a simple Hilbert--Lie algebra $\mathfrak{k}$ an explicit isomorphism from $\mathfrak{g}$ to one of the standard affinisations of $\mathfrak{k}$. 
The existence of such an isomorphism could also be derived from the classification 
of locally affine root systems, but 
for representation theoretic purposes it is crucial to obtain it explicitely 
as a deformation between two twists which is compatible 
with the root decompositions. 
We illustrate this by applying our isomorphism theorem to the study of positive energy highest weight representations of $\mathfrak{g}$. 

In subsequent work, the present paper will be used to obtain a complete classification 
of the positive energy highest weight representations of affinisations of $\mathfrak{k}$. 
}\end{abstract}

\maketitle

\section{Introduction}
Locally affine Lie algebras (LALAs) are natural generalisations of both affine Kac--Moody algebras and split locally finite Lie algebras. They were first introduced in \cite{MY06} as a subclass of the so-called \emph{locally extended affine Lie algebras} (LEALAs), and were later classified up to isomorphism in \cite{MY15} (see also \cite{Ne09}).
The LALAs roughly correspond to double extensions of (twisted) loop algebras over locally finite simple split Lie algebras (algebraic point of view), or equivalently, over simple Hilbert--Lie algebras $\kk$ (analytic point of view) -- in the latter case, we call such a LALA an \emph{affinisation} of $\kk$. The LALAs possess a root space decomposition with respect to some maximal abelian subalgebra, whose corresponding root system is a so-called \emph{locally affine root system} (LARS). The LARS were classified in \cite{YY08}, and those of infinite rank fall into $7$ distinct families of isomorphism classes, parametrised by the types $A_J^{(1)}$, $B_J^{(1)}$, $C_J^{(1)}$, $D_J^{(1)}$, $B_J^{(2)}$, $C_J^{(2)}$ and $BC_J^{(2)}$ for some infinite set $J$. To each of these types corresponds a ``minimal" affinisation of some simple Hilbert--Lie algebra $\kk$, which we call \emph{standard}.

In this paper, we give for each affinisation $\g$ of a simple Hilbert--Lie algebra $\kk$ an explicit isomorphism from $\g$ to one of the standard affinisations of $\kk$. The existence of such an isomorphism can also be derived from the classification of locally affine root systems, 
but for representation theoretic purposes 
it is crucial to obtain it in an explicit form. We illustrate this by applying our isomorphism theorem to the study of positive energy highest weight representations of $\g$, building on previous results from \cite{Ne09} on unitary highest weight representations of LALAs. 

Note that our results do not rely on the classification of LALAs from \cite{MY15} or of LARS from \cite{YY08}.

\medskip

We now present the main results of this paper in more detail, refering to Sections~\ref{section:HLA} and \ref{section:AOHLA} below for a more thorough account of the concepts presented. 
A \emph{Hilbert--Lie algebra} $\kk$ is a real Lie algebra as well as a real Hilbert space, whose scalar product $\langle\cdot,\cdot\rangle$ is invariant under the adjoint action, that is, such that $\langle [x,y],z\rangle =\langle x,[y,z]\rangle$ for all $x,y,z\in\kk$. By a theorem of Schue (\cite{Schue61}), any infinite-dimensional simple Hilbert--Lie algebra is isomorphic to the space $\kk:=\uu_2(\HH_{\KK})$ of skew-symmetric Hilbert-Schmidt operators on some Hilbert space $\HH_{\KK}$ over $\KK\in\{\RR,\CC,\HHH\}$. Any maximal abelian subalgebra $\ttt$ of $\kk$ (a \emph{Cartan subalgebra} of $\kk$) yields a root space decomposition $$\kk_{\CC}=\ttt_{\CC}\oplus\widehat{\bigoplus}_{\alpha\in\Delta}{\kk_{\CC}^{\alpha}}$$ of the complexification of $\kk$, whose corresponding set of roots $\Delta=\Delta(\kk,\ttt)\subseteq i\ttt^*$ is an irreducible locally finite root system of infinite rank, hence of one of the types $A_J$, $B_J$, $C_J$ or $D_J$ 
(see \cite{NeSt01} and \cite{LN04}). More precisely, if $\KK=\CC$ or $\KK=\HHH$, then $\kk$ possesses only one conjugacy class of Cartan subalgebras and $\Delta=A_J$ or $\Delta=C_J$, respectively. If $\KK=\RR$, then $\kk$ possesses two conjugacy classes of Cartan subalgebras, yielding the root systems $\Delta=B_J$ and $\Delta=D_J$.

Given an automorphism $\varphi\in\Aut(\kk)$ of finite order $N\in\NN$, there is a Cartan subalgebra $\ttt$ of $\kk$ such that $\ttt^{\varphi}$ is maximal abelian in $\kk^{\varphi}=\{x\in\kk \ | \ \varphi(x)=x\}$. In particular, $\kk_{\CC}$ possesses a $\varphi$-invariant root space decomposition $$\kk_{\CC}=\ttt^{\varphi}_{\CC}\oplus\widehat{\bigoplus}_{\alpha\in\Delta_{\varphi}}{\kk_{\CC}^{\alpha}}$$ with respect to $\ttt_0:=\ttt^{\varphi}$, with corresponding root system $\Delta_{\varphi}=\Delta(\kk,\ttt_0)\subseteq i\ttt_0^*$.

The \emph{$\varphi$-twisted loop algebra}\footnote{In the literature, the $\varphi$-twisted loop algebra is also defined as a space of $2\pi$-periodic smooth maps, instead of $2\pi N$-periodic smooth maps as in our definition. The reparametrisation needed to pass from one definition to the other is detailed in Remark~\ref{remark:conventions} below.} on $\kk$ is defined as $$\LLL_{\varphi}(\kk)=\big\{\xi\in C^{\infty}(\RR,\kk) \ \big| \ \xi(t+2\pi)=\varphi\inv(\xi(t)) \ \forall t\in\RR\big\}.$$
The scalar product $\langle \cdot,\cdot\rangle$ on $\kk$ can be extended to a non-degenerate invariant symmetric bilinear form on $\LLL_{\varphi}(\kk)$ by setting $$\langle\xi,\eta\rangle:=\frac{1}{2\pi }\int_{0}^{2\pi}{\langle\xi(t),\eta(t)\rangle \ dt}.$$
We  denote by $\langle\cdot,\cdot\rangle$ the hermitian extension of this scalar product to the complexification of $\LLL_{\varphi}(\kk)$.

Let $\der_0(\LLL_{\varphi}(\kk),\langle\cdot,\cdot\rangle)$ denote the space of skew-symmetric derivations $D$ of $\LLL_{\varphi}(\kk)$ that are \emph{diagonal}, in the sense that $D$ preserves each space $e^{int/N}\otimes \kk_{\CC}^{\alpha}$ for $n\in\ZZ$ and $\alpha\in\Delta_{\varphi}$.
Let $D_0\in \der_0(\LLL_{\varphi}(\kk),\langle\cdot,\cdot\rangle)$ be defined by $D_0(\xi):=\xi'$. Given a weight $\nu\in i\ttt_0^*$, 
we also define the derivation $\overline{D}_{\nu}$ of $\kk_{\CC}$ by setting $$\overline{D}_{\nu}(x)=i\nu(\alpha^{\sharp})x\quad\textrm{for all $x\in\kk_{\CC}^{\alpha}$, $\alpha\in\Delta_{\varphi}$,}$$ 
where $\alpha^{\sharp}\in i\ttt_0$ is such that $\langle h,\alpha^{\sharp}\rangle=\alpha(h)$ for all $h\in\ttt_0$. Then $\overline{D}_{\nu}$ restricts to a skew-symmetric derivation of $\kk$, which we extend to a diagonal derivation of $\LLL_{\varphi}(\kk)$ by setting
$\overline{D}_{\nu}(\xi)(t):=\overline{D}_{\nu}(\xi(t))$ for all $\xi\in\LLL_{\varphi}(\kk)$ and $t\in\RR$. Note that $\der_0(\LLL_{\varphi}(\kk),\langle\cdot,\cdot\rangle)$ is spanned by $D_0$ and all such $\overline{D}_{\nu}$ (see \cite[Theorem~7.2 and Lemma~8.6]{MY15}).
We set $$D_{\nu}:=D_0+\overline{D}_{\nu}\in \der_0(\LLL_{\varphi}(\kk),\langle\cdot,\cdot\rangle).$$
Then $D_{\nu}$ defines a $2$-cocycle $\omega_{D_{\nu}}(x,y)=\langle D_{\nu}x,y\rangle$ on $\LLL_{\varphi}(\kk)$, and extends to a derivation $\widetilde{D}_{\nu}(z,x):=(0,D_{\nu}x)$ of the corresponding central extension $\RR\oplus_{\omega_{D_{\nu}}}\LLL_{\varphi}(\kk)$. The double extension
$$\widehat{\LLL}_{\varphi}^{\nu}(\kk):=(\RR\oplus_{\omega_{D_{\nu}}}\LLL_{\varphi}(\kk))\rtimes_{\widetilde{D}_{\nu}}\RR$$
is called the (\emph{$\nu$-slanted and $\varphi$-twisted}) \emph{affinisation} of the Hilbert--Lie algebra $(\kk,\langle\cdot,\cdot\rangle)$. If $\nu=0$, we also simply write $\widehat{\LLL}_{\varphi}(\kk)=\widehat{\LLL}_{\varphi}^{\nu}(\kk)$.
The Lie bracket on $\widehat{\LLL}_{\varphi}^{\nu}(\kk)$ is given by 
$$[(z_1,x_1,t_1),(z_2,x_2,t_2)]=(\omega_{D_{\nu}}(x_1,x_2), [x_1,x_2]+t_1D_{\nu}x_2-t_2D_{\nu}x_1,0).$$
The subalgebra $\ttt_0^e:=\RR\oplus\ttt_0\oplus\RR$ is maximal abelian in $\widehat{\LLL}_{\varphi}^{\nu}(\kk)$ and yields a root space decomposition of $\widehat{\LLL}_{\varphi}^{\nu}(\kk)_{\CC}$ with corresponding set of roots $\widehat{\Delta}_{\varphi}=\Delta(\widehat{\LLL}_{\varphi}^{\nu}(\kk),\ttt_0^e)\subseteq i(\ttt_0^e)^*$. One can realise $\widehat{\Delta}_{\varphi}$ as a subset of $(\Delta_{\varphi}\cup\{0\})\times\ZZ$, where to $(\alpha,n)\in\widehat{\Delta}_{\varphi}$ corresponds the root space $e^{int/N}\otimes \kk_{\CC}^{(\alpha,n)}$ with $$\kk_{\CC}^{(\alpha,n)}:=\{x\in\kk_{\CC}^{\alpha} \ | \ \varphi\inv(x)=e^{2in\pi/N}x\}.$$ 
The set $(\widehat{\Delta}_{\varphi})_c:=\widehat{\Delta}_{\varphi}\cap (\Delta_{\varphi}\times\ZZ)$ of \emph{compact roots} of $\widehat{\Delta}_{\varphi}$ is then a LARS, hence isomorphic to one of the root systems $X_J^{(1)}$ or $Y_J^{(2)}$, for $X\in\{A,B,C,D\}$ and $Y\in\{B,C,BC\}$. The root systems of type $X_J^{(1)}$ can be realised as $\Delta(\widehat{\LLL}_{\psi}(\kk),\ttt_0^e)_c$ with $\psi=\id$ and $\kk$ such that $\Delta(\kk,\ttt)=X_J$. The three root systems of type $Y_J^{(2)}$ can be realised as $\Delta(\widehat{\LLL}_{\psi}(\kk),\ttt_0^e)_c$ for some order $2$ automorphism $\psi$ of $\kk$ and some suitable $\kk$. The corresponding three automorphisms $\psi$ are described in \cite[\S 2.2]{Hloopgroups} (see also Section~\ref{section:root_data}) 
and are called \emph{standard}. The above seven Lie algebras $\widehat{\LLL}_{\psi}(\kk)$ are also called \emph{standard}.

Here is the announced isomorphism theorem.
\begin{thmintro}\label{thm:mainintro}
Let $\kk$ be a simple Hilbert--Lie algebra and let $\varphi\in\Aut(\kk)$ be of finite order $N\in\NN$. Then there exist
\begin{itemize}
\item
an automorphism $\psi\in\Aut(\kk)$ which is either the identity or standard,
\item
a smooth one-parameter group $(\phi_t)_{t\in\RR}$ of automorphisms of $\kk$ commuting with $\varphi$ such that $\varphi=\phi_1\psi$,
\item
a maximal abelian subalgebra $\ttt$ of $\kk$ such that $\ttt^{\varphi}=\ttt^{\psi}=:\ttt_0$ and such that $\ttt_0$ is maximal abelian in both $\kk^{\varphi}$ and $\kk^{\psi}$,
\item
and a linear functional $\mu\in i\ttt_0^*$, 
\end{itemize}
such that, for any $\nu\in i\ttt_0^*$, the following assertions hold:
\begin{enumerate}
\item[(i)]
The map $$\widehat{\Phi}\co \widehat{\LLL}_{\varphi}^{\nu}(\kk)\to \widehat{\LLL}_{\psi}^{\mu+\nu}(\kk):(z_1,\xi(t),z_2)\mapsto (z_1,\phi_t(\xi(t)),z_2)$$ is an isomorphism of Lie algebras fixing the Cartan subalgebra
$\ttt_0^e:=\RR\oplus\ttt_0\oplus\RR$ pointwise.
\item[(ii)]
$\widehat{\Phi}$ induces an isomorphism of locally affine root systems given by
$$\pi\co \Delta(\widehat{\LLL}_{\varphi}^{\nu}(\kk),\ttt_0^e)_c\to 
\Delta(\widehat{\LLL}_{\psi}^{\mu+\nu}(\kk),\ttt_0^e)_c:(\alpha,n)\mapsto 
\Big(\alpha,N_{\psi}\cdot\big(\tfrac{n}{N}-\mu(\alpha^{\sharp})\big)\Big),$$
where $N_{\psi}\in\{1,2\}$ is the order of $\psi$.
\item[(iii)]
The Weyl groups of $\widehat{\LLL}_{\varphi}^{\nu}(\kk)$ and $\widehat{\LLL}_{\psi}^{\mu+\nu}(\kk)$ with respect to $\ttt_0^e$ coincide.
\end{enumerate}
\end{thmintro}

For each given pair $(\kk,\varphi)$, the parameters $\psi$, $(\phi_t)_{t\in\RR}$, $\ttt$ and $\mu$ whose existence is asserted in Theorem~\ref{thm:mainintro} are described explicitely in Section~\ref{section:root_data} below. The proof of Theorem~\ref{thm:mainintro} can be found at the end of Section~\ref{section:root_data}. 

Along the proof, we obtain an explicit description of the structure of finite order 
antiunitary operators on complex Hilbert spaces which may be 
of independent interest (see Proposition~\ref{prop:structure_antiunitary} below).

\medskip

We next state an application of our results to positive energy highest weight representations of $\g_{\nu}:=\widehat{\LLL}_{\varphi}^{\nu}(\kk)$. 
Let $\la\in i(\ttt_0^e)^*$ be an \emph{integral} weight of $\g_{\nu}$, in the sense that $\la$ takes integral values on all coroots $(\alpha,n)^{\vee}$, $(\alpha,n)\in (\widehat{\Delta}_{\varphi})_c$ (cf. \S\ref{subsection:RD} below). Assume moreover that $\la(\bc)\neq 0$, where $\bc:=(i,0,0)\in i\ttt_0^e$.
It then follows from \cite[Theorem~4.10]{Ne09} that $\g_{\nu}$ admits 
an integrable (irreducible) highest-weight representation
$$\rho_{\la}=\rho_{\la}^{\nu}\co  \g_{\nu}\to\End(L_{\nu}(\la))$$
with highest weight $\la$ and highest weight vector $v_{\la}\in L_{\nu}(\la)$.
In fact, $\rho_{\la}$ is even unitary with respect to some 
inner product on $L_{\nu}(\la)$ which is uniquely determined up to a positive factor (see \cite[Theorem~4.11]{Ne09}).

Let $\nu'\in i\ttt_0^*$, and extend the derivation $D_{\nu'}=D_0+\overline{D}_{\nu'}$ of $\LLL_{\varphi}(\kk)\subseteq\g_{\nu}$ to a skew-symmetric derivation of $\g_{\nu}$ by requiring that $D_{\nu'}(\ttt_0^e)=\{0\}$. Then $\rho_{\la}$ can be extended to a representation
$$\widetilde{\rho}_{\la}=\widetilde{\rho}_{\la}^{\thinspace\nu,\nu'}\co \g_{\nu}\rtimes \RR D_{\nu'}\to \End(L_{\nu}(\la))$$
of the semi-direct product $\g_{\nu}\rtimes \RR D_{\nu'}$ such that $\widetilde{\rho}_{\la}(D_{\nu'})$ annihilates the highest weight vector $v_{\la}$.
The representation $\widetilde{\rho}_{\la}$ is said to be of \emph{positive energy} if the spectrum of $H_{\nu'}:=-i\widetilde{\rho}_{\la}(D_{\nu'})$ is bounded from below. If this is the case, the infimum of the spectrum of $H_{\nu'}$ is called the \emph{minimal energy level} of $\widetilde{\rho}_{\la}$.

In the following theorem, we identify $i\ttt_0^*$ with the subspace $\{\mu\in i(\ttt_0^e)^* \ | \ \mu(\bc)=\mu(\bd)=0\}$ of $i(\ttt_0^e)^*$, where $\bd:=(0,0,-i)\in i\ttt_0^e$.
\begin{thmintro}\label{thm:PECintro}
Let $(\widehat{\LLL}_{\varphi}^{\nu}(\kk),\ttt_0^e)$ be as above, and let $\la\in i(\ttt_0^e)^*$ be an integral weight with $\la(\bc)\neq 0$. Let $\mu\in i\ttt_0^*$ and $\psi\in\Aut(\kk)$ of order $N_{\psi}\in\{1,2\}$ be the parameters provided by Theorem~\ref{thm:mainintro}, and denote by $\widehat{\WW}_{\psi}\subseteq\GL(i(\ttt_0^e)^*)$ and $\widehat{\Delta}_{\psi}\subseteq i(\ttt_0^e)^*$ the Weyl group and root system of the standard affinisation $\widehat{\LLL}_{\psi}(\kk)$ of $\kk$.
Then for any $\nu,\nu'\in i\ttt_0^*$, the following assertions are equivalent:
\begin{enumerate}
\item[(i)]
The highest weight representation $\widetilde{\rho}_{\la}\co \widehat{\LLL}_{\varphi}^{\nu}(\kk)\rtimes \RR D_{\nu'}\to \End(L_{\nu}(\la))$ is of positive energy.
\item[(ii)]
$M_{\mu,\nu,\nu'}:=\inf \chi(\widehat{\WW}_{\psi}.\la_{\mu+\nu}-\la_{\mu+\nu})>-\infty$, where $$\la_{\mu+\nu}:=\la-\la(\bc)(\mu+\nu)\in i(\ttt_0^e)^*\quad\textrm{and}\quad\chi\co\ZZ[\widehat{\Delta}_{\psi}]\to\RR:(\alpha,n)\mapsto (\mu+\nu')(\alpha^{\sharp})+n/N_{\psi}.$$
\end{enumerate}
Moreover, if $M_{\mu,\nu,\nu'}>-\infty$, then $M_{\mu,\nu,\nu'}$ is the minimal energy level of $\widetilde{\rho}_{\la}$.
\end{thmintro}
The proof of Theorem~\ref{thm:PECintro} can be found at the end of Section~\ref{section:COWG}.
Note that the ``standard" Weyl groups $\widehat{\WW}_{\psi}$ were given an explicit description in \cite[\S 3.4]{convexhull}, making the computation of $M_{\mu,\nu,\nu'}$ in the above theorem tractable. Using Theorem~\ref{thm:PECintro}, we will give in \cite{PEClocaff} a characterisation 
of all pairs $(\nu,\nu')$ yielding a positive energy representation $\widetilde{\rho}_{\la}^{\thinspace\nu,\nu'}$ as above, analoguous to the characterisation obtained in \cite{PEClocfin} for positive energy highest weight representations of locally finite split simple Lie algebras.

\section{Hilbert--Lie algebras}\label{section:HLA}
The general reference for this section is \cite[Section~1]{Hloopgroups}.

\subsection{Hilbert--Lie algebras}
A \emph{Hilbert--Lie algebra} $\kk$ is a real Lie algebra endowed with the structure of a real Hilbert space such that the scalar product $\langle \cdot,\cdot\rangle$ is invariant under the adjoint action, that is,
$$\langle [x,y],z\rangle=\langle x,[y,z]\rangle \quad\textrm{for all $x,y,z\in\kk$}.$$ 
By a theorem of Schue (\cite{Schue61}), every simple infinite-dimensional Hilbert--Lie algebra is isomorphic to 
$$\uu_2(\HH):=\{x\in B_2(\HH) \ | \ x^*=-x\}$$ for some infinite-dimensional Hilbert space $\HH$ over $\KK\in\{\RR,\CC,\HHH\}$, with scalar product given by
$$\langle x,y\rangle =\mathrm{tr}_{\RR}(xy^*)=-\mathrm{tr}_{\RR}(xy).$$
Here $\gl_2(\HH)=B_2(\HH)$ denotes the space of Hilbert-Schmidt operators on $\HH$. Note that if $\KK=\CC$, the complex conjugation on $\gl_2(\HH)$ is given by $\sigma(x)=-x^*$, and hence $\gl_2(\HH)$ can be identified with the complexification $\kk_{\CC}$ of $\kk:=\uu_2(\HH)$.

\subsection{Root decomposition}\label{subsection:RD}
Let $\g$ be a real Lie algebra and let $\g_{\CC}$ be its complexification, with complex conjugation $\sigma$ fixing $\g$ pointwise. Write $x^*:=-\sigma(x)$ for $x\in\g_{\CC}$, so that
$\g=\{x\in\g_{\CC} \ | \ x^*=-x\}$. Let $\ttt\subseteq\g$ be a maximal abelian subalgebra (a \emph{Cartan subalgebra}) with complexification $\ttt_{\CC}\subseteq\g_{\CC}$. For a linear functional $\alpha\in\ttt_{\CC}^*$, let
$$\g_{\CC}^{\alpha}:=\{x\in\g_{\CC} \ | \ [h,x]=\alpha(h)x \ \forall h\in\ttt_{\CC}\}$$
denote the corresponding \emph{root space}. Let also
$$\Delta:=\Delta(\g,\ttt):=\{\alpha\in\ttt_{\CC}^*\setminus\{0\} \ | \ \g_{\CC}^{\alpha}\}$$ 
be the \emph{root system} of $\g$ with respect to $\ttt$. Then $\g_{\CC}^0=\ttt_{\CC}$ and $[\g_{\CC}^{\alpha},\g_{\CC}^{\beta}]\subseteq \g_{\CC}^{\alpha+\beta}$ for all $\alpha,\beta\in \Delta\cup\{0\}$.

Assume that $\g$ is the Lie algebra of a group $G$ with an exponential function. Then $\ttt$ is called \emph{elliptic} if the subgroup $e^{\ad \ttt}=\Ad(\exp\ttt)\subseteq\Aut(\g)$ is equicontinuous. This implies in particular that $$\alpha\in i\ttt^*=\{\beta\in\ttt_{\CC}^* \ | \ \beta(\ttt)\subseteq i\RR\}\quad\textrm{for all $\alpha\in\Delta$},$$ 
and hence that $$\sigma(\g_{\CC}^{\alpha})=\g_{\CC}^{-\alpha}\quad\textrm{for all $\alpha\in\Delta$}.$$

A root $\alpha\in\Delta$ is called \emph{compact} if $\g_{\CC}^{\alpha}=\CC x_{\alpha}$ is one-dimensional and $\alpha([x_{\alpha},x_{\alpha}^*])>0$, so that
$$\sppan_{\CC}\{x_{\alpha},x_{\alpha}^*,[x_{\alpha},x_{\alpha}^*]\}\cap\g\cong \mathfrak{su}_2(\CC).$$
We denote by $\Delta_c$ the set of compact roots. If $\alpha\in\Delta_c$, there is a unique element $\check{\alpha}\in \ttt_{\CC}\cap [\g_{\CC}^{\alpha},\g_{\CC}^{-\alpha}]$ with $\alpha(\check{\alpha})=2$, called the \emph{coroot} of $\alpha$. Note that $\check{\alpha}\in i\ttt$. The \emph{Weyl group} $\WW=\WW(\g,\ttt)$ of $(\g,\ttt)$ is the subgroup of $\GL(\ttt)$ generated by the reflections
$$r_{\alpha}(x):=x-\alpha(x)\check{\alpha}\quad\textrm{for $\alpha\in\Delta_c$.}$$

\subsection{Locally finite root systems}\label{subsection:LFRS}
Let $\HH_{\KK}$ be some infinite-dimensional Hilbert space over $\KK\in\{\RR,\CC,\HHH\}$, and let $\kk=\uu_2(\HH_{\KK})$ be the corresponding simple Hilbert--Lie algebra, with invariant scalar product $\langle\cdot,\cdot\rangle$.
Let $\ttt\subseteq\kk$ be a maximal abelian subalgebra. It then follows from \cite{Schue61} that $\ttt$ is elliptic and that $\ttt_{\CC}\subseteq \kk_{\CC}$ defines a root space decomposition
$$\kk_{\CC}=\ttt_{\CC}\oplus\widehat{\bigoplus}_{\alpha\in\Delta}{\kk_{\CC}^{\alpha}}$$
which is a Hilbert space direct sum with respect to the hermitian extension, again denoted $\langle \cdot,\cdot\rangle$, of the scalar product to $\kk_{\CC}$. Moreover, all roots in $\Delta=\Delta(\kk,\ttt)\subseteq i\ttt^*$ are compact. In addition, there is an orthonormal basis $\BBB=\{E_j \ | \ j\in J\}\subseteq i\ttt$ of $\ttt_{\CC}\cong \ell^2(J,\CC)$ consisting of diagonal operators with respect to some orthonormal basis $\{e_j \ | \ j\in J'\}$ of $\HH_{\KK}$ (or of bloc-diagonal operators with $2\times 2$ blocs if $\KK=\RR$), such that $\BBB$
contains all coroots $\check{\alpha}$ ($\alpha\in\Delta$) in its $\ZZ$-span,  such that $\Delta$ is contained in the $\ZZ$-span of the linearly independent system $\{\epsilon_j \ | \ j\in J\}\subseteq i\ttt^*$ defined by $\epsilon_j(E_k)=\delta_{jk}$, and such that $\Delta$ is one of the following four infinite irreducible locally finite root systems of type $A_J$, $B_J$, $C_J$ or $D_J$:
\begin{equation*}
\begin{aligned}
A_J&:=\{\epsilon_j-\epsilon_k \ | \ j,k\in J, \ j\neq k\},\\
B_J&:=\{\pm \epsilon_j, \pm\epsilon_j\pm\epsilon_k \ | \ j,k\in J, \ j\neq k\},\\
C_J&:=\{\pm 2\epsilon_j, \pm\epsilon_j\pm\epsilon_k \ | \ j,k\in J, \ j\neq k\},\\
D_J&:=\{\pm\epsilon_j\pm\epsilon_k \ | \ j,k\in J, \ j\neq k\}.\\
\end{aligned}
\end{equation*}
If $\KK=\CC$ or $\KK=\HHH$, then $\kk$ possesses only one conjugacy class of Cartan subalgebras and $\Delta=A_J$ or $\Delta=C_J$, respectively (see \cite[Examples~1.10 and 1.12]{Hloopgroups}). If $\KK=\RR$, then $\kk$ possesses two conjugacy classes of Cartan subalgebras, yielding the root systems $\Delta=B_J$ and $\Delta=D_J$ (see \cite[Example~1.13]{Hloopgroups}). The above root data for $\kk$ will be described in more detail in Section~\ref{section:root_data} below.

Set $$\widehat{\ttt}:=\Big\{\sum_{j\in J}{x_jE_j} \ | \ x_j\in i\RR\Big\}\subseteq\uu(\HH_{\KK}^0),$$
where 
$$\HH_{\KK}^0:=\sppan_{\KK}\{e_j \ | \ j\in J'\}$$
is a pre-Hilbert-space with completion $\HH_{\KK}$. Note that any element of $\ttt$ is determined by its restriction to $\HH_{\KK}^0$; we will also view $\ttt$ as a subset of $\widehat{\ttt}$. The reason for this unusual convention is that we wish to define an inverse map for the injection of $\ttt$ in its dual $\ttt^*$ that is defined on the whole of $\ttt^*$.
More precisely, the assignment $\epsilon_j\mapsto E_j$, $j\in J$, defines an $\RR$-linear map $\sharp\co i\ttt^*\to i\widehat{\ttt}:\mu\mapsto\mu^{\sharp}$ such that 
$$\alpha(\mu^{\sharp}):=\langle\mu^{\sharp},\alpha^{\sharp}\rangle=\mu(\alpha^{\sharp})\quad\textrm{for all $\mu\in i\ttt^*$ and $\alpha\in\Delta$,}$$
where we have extended the scalar product $\langle\cdot,\cdot\rangle$ on $i\ttt\times i\ttt$ to $i\widehat{\ttt}\times i\ttt$.
For each $\alpha,\beta\in\Delta$, we set 
$$(\alpha,\beta):=\langle \alpha^{\sharp},\beta^{\sharp}\rangle.$$
Then $$\check{\alpha}=\tfrac{2}{(\alpha,\alpha)}\alpha^{\sharp}\quad\textrm{for all $\alpha\in\Delta$}.$$

\subsection{Automorphism groups}\label{subsection:AG}
Let $\kk=\uu_2(\HH_{\KK})$ for some infinite-dimensional Hilbert space $\HH_{\KK}$ over $\KK\in\{\RR,\CC,\HHH\}$. By \cite[Theorem~1.15]{Hloopgroups}, every automorphism $\varphi$ of $\kk$ is of the form
$$\varphi=\pi_A\co \kk\to\kk:x\mapsto AxA\inv,$$
for some unitary (for $\KK=\RR,\CC,\HHH$) or antiunitary (for $\KK=\CC$) operator $A$ on $\HH_{\KK}$. In particular, every automorphism of $\kk$ is isometric\footnote{In the classification theorem \cite[Theorem~1.15]{Hloopgroups}, the automorphisms $\varphi$ of $\kk$ are assumed to be isometric: this is used to show that if $\varphi(x_{\alpha})=\chi(\alpha)x_{\alpha}$ for $x_{\alpha}\in\kk_{\CC}^{\alpha}$ and a homomorphism $\chi\co \langle\Delta\rangle_{\mathrm{grp}}\to \CC^{\times}$, then $\im(\chi)\subseteq \mathbb{T}$. But this also follows from 
the fact that $\varphi$ preserves the real form $\kk$ of $\kk_\CC$.}.
We recall that an operator $A$ on $\HH_{\CC}$ is called \emph{antiunitary} if it is antilinear and satisfies 
$$\langle Ax,Ay\rangle =\overline{\langle x,y\rangle}\quad\textrm{for all $x,y\in\HH_{\CC}$.}$$

\section{Affinisations of Hilbert--Lie algebras}\label{section:AOHLA}
In this section, we let $(\kk,\langle\cdot,\cdot\rangle)$ be a simple Hilbert--Lie algebra and $\varphi$ be an automorphism of $\kk$ of finite order $N\in\NN$, and we set $\zeta:=e^{2i\pi/N}\in\CC$. The general reference for this section is \cite[Section~2]{Hloopgroups}.

\subsection{Finite order automorphisms}\label{subsection:FOA}
Let $\ttt_0$ be a maximal abelian subalgebra of 
\[ \kk^{\varphi}=\{x\in \kk \ | \ \varphi(x)=x\}.\]
 Then the centraliser in $\kk$ of $\ttt_0$ is a maximal abelian subalgebra $\ttt$ of $\kk$ by \cite[Lemma~D.2]{Hloopgroups}. Thus $\ttt_0=\ttt^{\varphi}=\ttt\cap \kk^{\varphi}$. 

Since $\varphi(\ttt)=\ttt$, it follows from \S\ref{subsection:LFRS} that the Lie algebra $\kk_{\CC}$ decomposes as an orthogonal direct sum of $\varphi$-invariant $\ttt^{\varphi}$-weight spaces $\kk_{\CC}^{\beta}$ for $\beta\in i(\ttt^{\varphi})^*$. Let
$$\Delta_{\varphi}:=\Delta(\kk,\ttt^{\varphi}):=\{\beta\in i(\ttt^{\varphi})^*\setminus\{0\} \ | \ \kk_{\CC}^{\beta}\neq\{0\}\}$$
denote the set of nonzero $\ttt^{\varphi}$-weight in $\kk$, and for each $n\in\ZZ$ and $\beta\in\Delta_{\varphi}\cup\{0\}$, set $$\kk_{\CC}^{(\beta,n)}:=\kk_{\CC}^{\beta}\cap\kk_{\CC}^n \quad\textrm{where}\quad \kk_{\CC}^n:=\{x\in\kk_{\CC} \ | \ \varphi\inv(x)=\zeta^nx\}.$$
Thus $$\kk_{\CC}^{\beta}=\sum_{n=0}^{N-1}{\kk_{\CC}^{(\beta,n)}}.$$
Moreover, $\dim \kk_{\CC}^{(\beta,n)}\leq 1$ for all $\beta\in\Delta_{\varphi}$ and $n\in\ZZ$ by \cite[Appendix~D]{Hloopgroups}.
For each $n\in\ZZ$, we let $\Delta_n\subseteq i(\ttt^{\varphi})^*$ denote the set of nonzero $\ttt^{\varphi}$-weights in $\kk_{\CC}^n$, that is, the set of $\beta\in\Delta_{\varphi}$ such that $\dim \kk_{\CC}^{(\beta,n)}=1$. Note that $\Delta_{\varphi}=\Delta_0=\Delta(\kk,\ttt)$ if $\varphi=\id$.

As $\big\langle \kk_{\CC}^{\alpha},\kk_{\CC}^{\beta}\big\rangle = \{0\}$ for all $\alpha,\beta\in \Delta_{\varphi}\cup\{0\}$ with $\alpha\neq\beta$, and as $\big\langle \kk_{\CC}^{m},\kk_{\CC}^{n}\big\rangle = \{0\}$ if $m+n\notin N\ZZ$, the restriction of $\langle\cdot,\cdot\rangle$ to $\ttt^{\varphi}_{\CC}=\kk_{\CC}^{(0,0)}$ is non-degenerate. In particular, the map $\sharp\co i\ttt^*\to i\widehat{\ttt}$ from \S\ref{subsection:LFRS} factors through a map 
$$\sharp\co i(\ttt^{\varphi})^*\to i\widehat{\ttt^{\varphi}}\subseteq i\widehat{\ttt}:\mu\mapsto\mu^{\sharp}$$
making the diagram
$$\begin{CD}
i\ttt^* @>\sharp>> i\widehat{\ttt}\\
@VVV @AAA\\
i(\ttt^{\varphi})^* @>\sharp>> i\widehat{\ttt^{\varphi}}
\end{CD}$$
commute. In other words, if $\mu\in i(\ttt^{\varphi})^*$, then $\mu^{\sharp}$ is the unique element of $i\widehat{\ttt^{\varphi}}$ satisfying $\langle \mu^{\sharp},h\rangle=\mu(h)$ for all $h\in i\ttt^{\varphi}$. As before, we set 
$$(\alpha,\beta):=\langle \alpha^{\sharp},\beta^{\sharp}\rangle\quad\textrm{for all $\alpha,\beta\in\Delta_{\varphi}$.}$$

For $x\in\kk_{\CC}^{(\beta,n)}$ ($\beta\in\Delta_{\varphi}$, $n\in\ZZ$), we have $[x,x^*]\in \kk_{\CC}^{(0,0)}=\ttt_{\CC}^{\varphi}$, and for $h\in\ttt_{\CC}^{\varphi}$,
$$\langle h,[x,x^*]\rangle =\langle [h,x],x\rangle=\beta(h)\langle x,x\rangle =\langle h,\langle x,x\rangle\beta^{\sharp}\rangle$$
and hence
$$[x,x^*]=\langle x,x\rangle \beta^{\sharp}.$$
In particular, choosing $x\in\kk_{\CC}^{(\beta,n)}$ such that $\langle x,x\rangle=\tfrac{2}{(\beta,\beta)}$, we may define as before the \emph{coroot} of $\beta$ as
$$\check{\beta}:=[x,x^*]= \frac{2}{(\beta,\beta)}\beta^{\sharp}.$$

\subsection{Loop algebras}\label{subsection:LA}
Consider the \emph{$\varphi$-twisted loop algebra} $$\LLL_{\varphi}(\kk)=\big\{\xi\in C^{\infty}(\RR,\kk) \ \big| \ \xi(t+2\pi)=\varphi\inv(\xi(t)) \ \forall t\in\RR\big\},$$ with Lie bracket
$[\xi,\eta](t)=[\xi(t),\eta(t)]$. If $\varphi=\id$, we simply write $\LLL(\kk):=\LLL_{\id}(\kk)$ for the corresponding untwisted loop algebra. We extend the scalar product $\langle \cdot,\cdot\rangle$ on $\kk$ to a non-degenerate invariant symmetric bilinear form on $\LLL_{\varphi}(\kk)$ by setting $$\langle\xi,\eta\rangle:=\frac{1}{2\pi}\int_{0}^{2\pi }{\langle\xi(t),\eta(t)\rangle \ dt}.$$
We again denote by $\langle \cdot,\cdot\rangle$ the unique hermitian extension of this form to  $\LLL_{\varphi}(\kk)_{\CC}$, and we write 
$$\sigma(\xi)(t):=-\xi^*(t)=-\xi(t)^*, \quad x\in \LLL_{\varphi}(\kk)_{\CC},$$
for the corresponding complex conjugation on $\LLL_{\varphi}(\kk)_{\CC}$.

Given $n\in\ZZ$ and $t\in\RR$, we set $$e_n(t):=e^{int/N},$$
so that $e_n\otimes x\in C^{\infty}(\RR,\kk)$ for all $x\in\kk$. 
Note then that for any $x\in\kk_{\CC}$, 
$$\xi:=e_n\otimes x\in\LLL_{\varphi}(\kk)_{\CC}\iff x\in\kk_{\CC}^n$$
because $\xi(t+2\pi)=\zeta^n\xi(t)$ and $\varphi\inv(\xi(t))=e_n(t)\otimes\varphi\inv(x)$.

\subsection{Derivations}\label{subsection:D}
Let $\der(\LLL_{\varphi}(\kk),\langle\cdot,\cdot\rangle)$ denote the space of derivations $D$ of $\LLL_{\varphi}(\kk)$ that are skew-symmetric with respect to $\langle\cdot,\cdot\rangle$, that is, such that $\langle D\xi,\eta\rangle=-\langle \xi,D\eta\rangle$ for all $\xi,\eta\in\LLL_{\varphi}(\kk)$. Let $D_0\in \der(\LLL_{\varphi}(\kk),\langle\cdot,\cdot\rangle)$ be defined by $$D_0(\xi)=\xi'\quad\textrm{for all $\xi\in \LLL_{\varphi}(\kk)$.}$$ Given $\mu\in i(\ttt^{\varphi})^*$, we also define the derivation $\overline{D}_{\mu}$ of $\kk_{\CC}$ by setting $$\overline{D}_{\mu}(x)=i\mu(\alpha^{\sharp})x\quad\textrm{for all $x\in\kk_{\CC}^{\alpha}$, $\alpha\in\Delta_{\varphi}$.}$$ 
Since $\mu(\alpha^{\sharp})\in \RR$ for all $\alpha\in\Delta_{\varphi}$, $\overline{D}_{\mu}$ restricts to a skew-symmetric derivation of $\kk$. 
Note that $\overline{D}_{\mu}$ commutes with $\varphi$ since it stabilises each $\kk_{\CC}^{(\beta,n)}$ for $\beta\in\Delta_{\varphi}$, $n\in\ZZ$. Hence it extends to a skew-symmetric derivation of $\LLL_{\varphi}(\kk)$ by setting
$$\overline{D}_{\mu}(\xi)(t):=\overline{D}_{\mu}(\xi(t))\quad\textrm{for all $\xi\in\LLL_{\varphi}(\kk)$ and $t\in\RR$.}$$
Finally, we set $D_{\mu}:=D_0+\overline{D}_{\mu}\in \der(\LLL_{\varphi}(\kk),\langle\cdot,\cdot\rangle)$, so that 
\begin{equation}\label{eqn:Dmu}
D_{\mu}(e_n\otimes x)=i\big(\tfrac{n}{N}+\mu(\alpha^{\sharp})\big)(e_n\otimes x)
\end{equation}
for all $\alpha\in\Delta_{\varphi}$ and $x\in\kk_{\CC}^{(\alpha,n)}$.

\subsection{Double extensions}\label{subsection:DE}
We define on $\LLL_{\varphi}(\kk)$ the $2$-cocycle
$$\omega_{D_{\mu}}(x,y)=\langle D_{\mu}x,y\rangle\quad\textrm{for all $x,y\in \LLL_{\varphi}(\kk)$}.$$
Let $\RR\oplus_{\omega_{D_{\mu}}}\LLL_{\varphi}(\kk)$ be the corresponding central extension, with Lie bracket
$$[(z_1, x_1), (z_2, x_2)]=(\omega_{D_{\mu}}(x_1,x_2), [x_1,x_2]).$$
Extend $D_{\mu}$ to a derivation $\widetilde{D}_{\mu}$ of $\RR\oplus_{\omega_{D_{\mu}}}\LLL_{\varphi}(\kk)$ by 
$$\widetilde{D}_{\mu}(z,x):=(0,D_{\mu}x).$$
Let $$\g:=\widehat{\LLL}_{\varphi}^{\mu}(\kk):=(\RR\oplus_{\omega_{D_{\mu}}}\LLL_{\varphi}(\kk))\rtimes_{\widetilde{D}_{\mu}}\RR$$
denote the corresponding double extension, with Lie bracket
$$[(z_1,x_1,t_1),(z_2,x_2,t_2)]=(\omega_{D_{\mu}}(x_1,x_2), [x_1,x_2]+t_1D_{\mu}x_2-t_2D_{\mu}x_1,0).$$
The Lie algebra $\g$ is called the (\emph{$\mu$-slanted and $\varphi$-twisted}) \emph{affinisation} of the Hilbert--Lie algebra $(\kk,\langle\cdot,\cdot\rangle)$. 
If $\varphi=\id_{\kk}$ (resp. $\mu=0$), we also simply write $\widehat{\LLL}^{\mu}(\kk)$ (resp. $\widehat{\LLL}_{\varphi}(\kk)$) instead of $\widehat{\LLL}_{\varphi}^{\mu}(\kk)$.
Note that in terms of the hermitian extension of $\langle \cdot,\cdot\rangle$ to $\LLL_{\varphi}(\kk)_{\CC}$, the Lie bracket on $\g_{\CC}$ is given by
$$[(z_1,x_1,t_1),(z_2,x_2,t_2)]=(\langle D_{\mu}x_1,-x_2^*\rangle, [x_1,x_2]+t_1D_{\mu}x_2-t_2D_{\mu}x_1,0).$$
One can endow $\g$ with the non-degenerate invariant symmetric bilinear form $\kappa\co \g\times\g\to\RR$ defined by
$$\kappa((z_1,x_1,t_1),(z_2,x_2,t_2))=\langle x_1,x_2\rangle+z_1t_2+z_2t_1,$$
thus turning $\g$ into a \emph{quadratic Lie algebra}.
Moreover, the subalgebra
$$\ttt_{\g}^{\varphi}:=\RR\oplus\ttt^{\varphi}\oplus\RR$$
is maximal abelian and elliptic in $\g$. The root system $\Delta_{\g}:=\Delta(\g,\ttt_{\g})$ can be identified with the set $$\Delta_{\g}=\{(\alpha,n) \ | \ n\in\ZZ, \ \alpha\in\Delta_n\cup\{0\}\}\setminus\{(0,0)\},$$ where
$$(\alpha,n)(z,h,t):=(0,\alpha,n)(z,h,t):=\alpha(h)+it\big(\tfrac{n}{N}+\mu(\alpha^{\sharp})\big).$$
For $(\alpha,n)\in\Delta_{\g}$, the corresponding root space is
$$\g_{\CC}^{(\alpha,n)}=e_n\otimes\kk_{\CC}^{(\alpha,n)}.$$
The root $(\alpha,n)$ is compact if and only if $\alpha\neq 0$. Hence
$$(\Delta_{\g})_c=\bigcup_{n=0}^{N-1}{\Delta_n\times(n+N\ZZ)}\subseteq \{0\}\times i(\ttt^{\varphi})^*\times\RR.$$

Given $n\in\ZZ$ and $x\in \kk_{\CC}^{(\alpha,n)}$ with $[x,x^*]=\check{\alpha}$ ($\alpha\in\Delta_n$), we deduce from (\ref{eqn:Dmu}) that 
$$\big\langle D_{\mu}(e_n\otimes x),-(e_{-n}\otimes x^*)^*\big\rangle=\big\langle i\big(\tfrac{n}{N}+\mu(\alpha^{\sharp})\big)e_n\otimes x,-e_n\otimes x\big\rangle =-i\big(\tfrac{n}{N}+\mu(\alpha^{\sharp})\big)\langle x,x\rangle.$$
Since $\langle x,x\rangle=2/(\alpha,\alpha)$, the element $e_n\otimes x\in\g_{\CC}^{(\alpha,n)}$ thus satisfies 
$$[e_n\otimes x,(e_{n}\otimes x)^*]=[e_n\otimes x,e_{-n}\otimes x^*]=\Big(\frac{-2i\big(\tfrac{n}{N}+\mu(\alpha^{\sharp})\big)}{(\alpha,\alpha)},\check{\alpha},0\Big).$$
As $(\alpha,n)$ has value $2$ on this element, we deduce that the coroot associated to $(\alpha,n)$ is
\begin{equation}\label{eqn:coroot}
(\alpha,n)^{\vee}=\Big(\frac{-2i\big(\tfrac{n}{N}+\mu(\alpha^{\sharp})\big)}{(\alpha,\alpha)},\check{\alpha},0\Big).
\end{equation}

\subsection{Locally affine root systems}\label{subsection:LARS}
We define on $\sppan_{\QQ}(\Delta_{\g})_c$ the positive semidefinite bilinear form $(\cdot,\cdot)$ by
$$\big((\alpha,m),(\beta,n)\big):=(\alpha,\beta)\quad\textrm{for all $m,n\in\ZZ$, $\alpha\in\Delta_m$ and $\beta\in\Delta_n$.}$$
Then the triple  $(\sppan_{\QQ}(\Delta_{\g})_c,(\Delta_{\g})_c,(\cdot,\cdot))$ is an irreducible reduced locally affine root system in the sense of \cite[Definition~2.4]{Ne09} (see also \cite{YY08}).

Such root systems have been classified (see \cite[Corollary~13]{YY08}) and those of infinite rank fall into $7$ distinct families of isomorphism classes, parametrised by the types $A_J^{(1)}$, $B_J^{(1)}$, $C_J^{(1)}$, $D_J^{(1)}$, $B_J^{(2)}$, $C_J^{(2)}$ and $BC_J^{(2)}$ for some infinite set $J$. 
Denoting by $\QQ^{(J)}$ the free $\QQ$-vector space with canonical basis $\{\epsilon_j \ | \ j\in J\}$ and scalar product $(\epsilon_j,\epsilon_k)=\delta_{jk}$, these can be realised in $\QQ^{(J)}\times\QQ$ as
\begin{equation*}
\begin{aligned}
X_J^{(1)}&:=X_J\times\ZZ\quad\textrm{for $X\in\{A,B,C,D\}$},\\
B_J^{(2)}&:=(B_J\times 2\ZZ)\cup \big(\{\pm\epsilon_j \ | \ j\in J\}\times (2\ZZ+1)\big),\\
C_J^{(2)}&:=(C_J\times 2\ZZ)\cup \big(D_J\times (2\ZZ+1)\big),\\
BC_J^{(2)}&:=(B_J\times 2\ZZ)\cup \big((B_J\cup C_J)\times (2\ZZ+1)\big),
\end{aligned}
\end{equation*}
where $A_J$, $B_J$, $C_J$ and $D_J$ are as in \S\ref{subsection:LFRS} and where the scalar product on $\QQ^{(J)}\times\QQ$ is given by $\big((\alpha,t),(\alpha',t')\big):=(\alpha,\alpha')$.

If $\Delta=\Delta(\kk,\ttt)$ has type $X_J$ for some $X\in\{A,B,C,D\}$ (see \S\ref{subsection:LFRS}), then the root system of type $X_J^{(1)}$ is obtained as the set of compact roots $$(\Delta_{\g})_c=\Delta_0\times\ZZ=\Delta\times\ZZ$$ of the untwisted doubly extended loop algebra $\g=\widehat{\LLL}(\kk)$ ($\varphi=\id_{\kk}$). 
The root system of type $X_J^{(2)}$ for $X\in\{B,C,BC\}$ can similarly be obtained as the set of compact roots 
$$(\Delta_{\g})_c=(\Delta_0\times 2\ZZ)\cup (\Delta_1\times(1+2\ZZ))$$
of a twisted doubly extended loop algebra $\g=\widehat{\LLL}_{\varphi}(\kk)$, for some suitable choice of a simple Hilbert--Lie algebra $\kk=\kk_X$ and of an automorphism $\varphi=\varphi_X\in\Aut(\kk)$ of order $2$. The three involutive automorphisms $\varphi_X$, $X\in\{B,C,BC\}$, are described in \cite[\S 2.2]{Hloopgroups} and are called \emph{standard} 
(see also Section~\ref{section:root_data}). We will also call the $7$ Lie algebras $\g$ described above \emph{standard affinisations} of the corresponding Hilbert--Lie algebra $\kk$.
We will describe these $7$ standard affinisations and the corresponding root data in more 
detail in Section~\ref{section:root_data} below.

\subsection{Weyl group}\label{subsection:WG}
For each $(\alpha,n)\in (\Delta_{\g})_c$, the reflection $r_{(\alpha,n)}\in\GL(\ttt_{\g}^{\varphi})$ is given by 
\begin{equation}
\begin{aligned}
r_{(\alpha,n)}(z,h,t)&=(z,h,t)-(\alpha,n)(z,h,t)\cdot (\alpha,n)^{\vee}\\
&=(z,h,t)-\big(\alpha(h)+it\big(\tfrac{n}{N}+\mu(\alpha^{\sharp})\big)\big)\cdot \Big(\frac{-2i(\tfrac{n}{N}+\mu(\alpha^{\sharp}))}{(\alpha,\alpha)},\check{\alpha},0\Big).\label{eqn:ralphan}
\end{aligned}
\end{equation}
In particular, 
\begin{equation*}
r_{(\alpha,0)}(z,h,t)=(z,h,t)-\big(\alpha(h)+it\mu(\alpha^{\sharp})\big)\cdot \big(-i\mu(\check{\alpha}),\check{\alpha},0\big).
\end{equation*}
Denote by 
$$\widehat{\WW}_{\mu}:=\WW(\g,\ttt_{\g}^{\varphi})=\big\langle r_{(\alpha,n)} \ | \ n\in\ZZ, \ \alpha\in\Delta_n\big\rangle \subseteq \GL(\ttt_{\g}^{\varphi})$$
the Weyl group of $(\g,\ttt_{\g}^{\varphi})$, and by $\WW_{\mu}$ the subgroup of $\widehat{\WW}_{\mu}$ generated by the reflections $r_{(\alpha,0)}$ for $\alpha\in\Delta_0$. Note that $\widehat{\WW}_{\mu}$ preserves the invariant bilinear form $\kappa\co\g\times\g\to\RR$.

For each $x\in\ttt^{\varphi}$, define the automorphism $\tau_x=\tau(x)\in\GL(\ttt_{\g}^{\varphi})$ by
\begin{equation*}
\tau_x(z,h,t)=\Big(z-\langle h,x\rangle-\frac{t\langle x,x\rangle}{2},h+tx,t\Big).
\end{equation*}
Then $\tau_{x_1}\tau_{x_2}=\tau_{x_1+x_2}$ for all $x_1,x_2\in \ttt^{\varphi}$. Moreover, defining for each $\alpha\in\Delta_{\varphi}$ and $n\in\ZZ$ the reflection $r_{(\alpha,n)}\in\GL(\ttt_{\g}^{\varphi})$ by the formula (\ref{eqn:ralphan}) even if $\alpha\notin \Delta_n$, one can check that
$r_{(\alpha,0)}r_{(\alpha,n)}=\tau_{in\check{\alpha}/N}$ (cf. \cite[\S 3.4]{convexhull}).

Assume now that for each $\alpha\in\Delta_{\varphi}$ there exists some $\beta\in\Delta_0$ such that $r_{(\alpha,0)}=r_{(\beta,0)}$. This is for instance the case if $(\Delta_{\g})_c$ is one of the $7$ locally affine root systems from \S\ref{subsection:LARS}.
Denoting by $\TTT_{\varphi}$ the abelian subgroup of $\ttt^{\varphi}$ generated by $\{in\check{\alpha}/N \ | \ n\in\ZZ,\ \alpha\in\Delta_n\}$, we deduce the following semi-direct decomposition of $\widehat{\WW}_{\mu}$ inside $\GL(\ttt_{\g}^{\varphi})$:
$$\widehat{\WW}_{\mu}=\tau(\TTT_{\varphi})\rtimes \WW_{\mu}.$$

We will describe in Section~\ref{section:COWG} an explicit isomorphism between $\widehat{\WW}_{\mu}$ and $\widehat{\WW}_{0}$.

\section{Isomorphisms of twisted loop algebras}
In this section, we fix some simple Hilbert--Lie algebra $\kk=\uu_2(\HH_{\KK})$ for some infinite-dimensional Hilbert space $\HH_{\KK}$ over $\KK\in\{\RR,\CC,\HHH\}$, as well as some automorphism $\varphi\in\Aut(\kk)$ of finite order.

\begin{lemma}\label{lemma:untwisting}
Let $\psi\in\Aut(\kk)$ be of finite order, and assume that there exists a smooth one-parameter group $(\phi_t)_{t\in\RR}$ of automorphisms of $\kk$ commuting with $\psi$ such that $\phi_1\psi=\varphi$.
Then the map $$\Phi\co \LLL_{\varphi}(\kk)\to \LLL_{\psi}(\kk):\xi\mapsto \Phi(\xi)(t):=\phi_{t/2\pi}(\xi(t))$$ is an isomorphism of Lie algebras.
\end{lemma}
\begin{proof}
Let $\xi\in\LLL_{\varphi}(\kk)$. Then $\Phi(\xi)\in C^{\infty}(\RR,\kk)$ and for all $t\in\RR$,
$$\Phi(\xi)(t+2\pi)=\phi_{\frac{t}{2\pi}+1}(\xi(t+2\pi))=\phi_{t/2\pi}\varphi\psi\inv(\varphi\inv(\xi(t)))=\psi\inv\phi_{t/2\pi}(\xi(t))=\psi\inv(\Phi(\xi)(t)).$$
Hence $\Phi(\LLL_{\varphi}(\kk))\subseteq \LLL_{\psi}(\kk)$. Moreover,
$$[\Phi(\xi),\Phi(\eta)](t)=[\varphi_{t/2\pi}(\xi(t)),\varphi_{t/2\pi}(\eta(t))]=\varphi_{t/2\pi}([\xi(t),\eta(t)])=\Phi([\xi,\eta])(t)$$
for all $\xi,\eta\in\LLL_{\varphi}(\kk)$ and $t\in\RR$, so that $\Phi$ is indeed a Lie algebra morphism.
Similarly, the map $\Phi\inv\co \LLL_{\psi}(\kk)\to \LLL_{\varphi}(\kk):\xi\mapsto \Phi\inv(\xi)(t):=\phi_{-t/2\pi}(\xi(t))$ is a well-defined Lie algebra morphism. Since it is an inverse for $\Phi$, the lemma follows.
\end{proof}

\begin{prop}\label{prop:untwisting_extension2}
Let $\psi\in\Aut(\kk)$ be of finite order, and let $(U_t)_{t\in\RR}$ be a smooth one-parameter group of unitary operators $U_t\in U(\HH_{\KK})$ such that the corresponding automorphisms $\phi_t=\pi_{U_t}$ of $\kk$ commute with $\psi$ and such that $\phi_1\psi=\varphi$. Let $\ttt$ be a maximal abelian subalgebra of $\kk$, and assume that $\ttt^{\varphi}=\ttt^{\psi}=:\ttt_0$. Assume moreover that $\ttt_0$ is maximal abelian in both $\kk^{\varphi}$ and $\kk^{\psi}$ and that $\ttt_0\subseteq\ttt^{\phi_t}$ for all $t\in\RR$.

Let $\mu,\nu\in i\ttt_0^*$ and assume that the skew symmetric operator $\Lambda_{\mu}:=i\mu^{\sharp}\in \uu(\HH_{\KK})$ satisfying $\langle \Lambda_{\mu},h\rangle = i\mu(h)$ for all $h\in i\ttt_0$ is bounded. Assume moreover that $$\frac{d}{dt}U_{t/2\pi}=-\Lambda_{\mu}U_{t/2\pi}=-U_{t/2\pi}\Lambda_{\mu}.$$ Then the following holds:
\begin{enumerate}
\item[(i)]
The isomorphism $\Phi\co \LLL_{\varphi}(\kk)\to \LLL_{\psi}(\kk)$ provided by Lemma~\ref{lemma:untwisting} extends to an isomorphism $$\widehat{\Phi}\co (\RR\oplus_{\omega_{D_{\nu}}}\LLL_{\varphi}(\kk))\rtimes_{\widetilde{D}_{\nu}}\RR\to (\RR\oplus_{\omega_{D_{\mu+\nu}}}\LLL_{\psi}(\kk))\rtimes_{\widetilde{D}_{\mu+\nu}}\RR$$ fixing 
$\ttt_0^e:=\RR\oplus\ttt_0\oplus\RR$ pointwise.
\item[(ii)]
$\widehat{\Phi}$ induces an isomorphism of locally affine root systems given by
$$\pi\co \Delta(\widehat{\LLL}_{\varphi}^{\nu}(\kk),\ttt_0^e)_c\to \Delta(\widehat{\LLL}_{\psi}^{\mu+\nu}(\kk),\ttt_0^e)_c:(\alpha,n)\mapsto \big(\alpha,N_{\psi}\cdot(\tfrac{n}{N_{\varphi}}-\mu(\alpha^{\sharp}))\big),$$
where $N_{\varphi}$ and $N_{\psi}$ are the respective orders of $\varphi$ and $\psi$.
\item[(iii)]
The Weyl groups $\WW(\widehat{\LLL}_{\varphi}^{\nu}(\kk),\ttt_0^e),\WW(\widehat{\LLL}_{\psi}^{\mu+\nu}(\kk),\ttt_0^e)\subseteq \GL(\ttt_0^e)$ coincide.
\end{enumerate}
\end{prop}
\begin{proof}
Write for short $\g_{\varphi}:=\widehat{\LLL}_{\varphi}^{\nu}(\kk)$ and $\g_{\psi}:=\widehat{\LLL}_{\psi}^{\mu+\nu}(\kk)$, as well as $\Delta_{\g_{\varphi}}$ and $\Delta_{\g_{\psi}}$ for the corresponding root systems with respect to the Cartan subalgebra $\ttt_0^e$.
Note that $\Delta_{\varphi}=\Delta(\kk,\ttt_0)=\Delta_{\psi}$. For each $n\in\ZZ$ let $\Delta_n^{\varphi}$ and $\Delta_n^{\psi}$ respectively denote the set of nonzero $\ttt_0$-weights on $$\kk_{\CC}^n(\varphi):=\{x\in\kk_{\CC} \ | \ \varphi\inv(x)=e^{2in\pi/N_{\varphi}}x\}\quad\textrm{and}\quad \kk_{\CC}^n(\psi):=\{x\in\kk_{\CC} \ | \ \psi\inv(x)=e^{2in\pi/N_{\psi}}x\},$$ as in \S\ref{subsection:FOA}. Thus
$$(\Delta_{\g_{\varphi}})_c=\bigcup_{0\leq n< N_{\varphi}}{\Delta_n^{\varphi}\times (n+N_{\varphi}\ZZ)}\quad\textrm{and}\quad (\Delta_{\g_{\psi}})_c=\bigcup_{0\leq n< N_{\psi}}{\Delta_n^{\psi}\times (n+N_{\psi}\ZZ)}$$
(cf. \S\ref{subsection:DE}).

We extend the isomorphism $\Phi\co \LLL_{\varphi}(\kk)\to \LLL_{\psi}(\kk)$ provided by Lemma~\ref{lemma:untwisting} to a bijective linear map $\widehat{\Phi}\co\g_{\varphi}\to \g_{\psi}$ by setting $\widehat{\Phi}(1,0,0):=(1,0,0)$ and $\widehat{\Phi}(0,0,1):=(0,0,1)$.

We first claim that $\overline{D}_{\mu}=\ad(\Lambda_{\mu})\in\der(\kk,\langle\cdot,\cdot\rangle)$. Indeed, let $\{E_j \ |  j\in J\}$ be some orthonormal basis of the real Hilbert space $i\ttt_0$ whose $\RR$-span contains all $\alpha^{\sharp}$, $\alpha\in \Delta_{\psi}$ (cf. \S\ref{subsection:LFRS} and \S\ref{subsection:FOA}). Write $$\Lambda_{\mu}=\sum_{j\in J}{\mu_jE_j}\quad\textrm{where $\mu_j:=\langle \Lambda_{\mu},E_j\rangle=i\mu(E_j)$ for all $j\in J$.}$$
Then for any $x$ be in the $\ttt_0$-weight space of $\kk_{\CC}$ corresponding to $\alpha\in\Delta_{\psi}\cup\{0\}$, we have
$$\ad(\Lambda_{\mu})(x)=\sum_{j\in J}{\mu_j[E_j,x]}=\sum_{j\in J}{\mu_j\alpha(E_j)x}=\sum_{j\in J}{\mu_j\langle E_j,\alpha^{\sharp}\rangle x}=\langle \Lambda_{\mu},\alpha^{\sharp}\rangle x=i\mu(\alpha^{\sharp})x=\overline{D}_{\mu}x,$$
as desired. Note that the above sums are finite because $\alpha^{\sharp}$ is a (finite) linear combination of the $E_j$.

We then obtain for all $\xi\in\LLL_{\varphi}(\kk)$ that
\begin{equation*}
\begin{aligned}
(\Phi(\xi))'(t)=\frac{d}{dt}(U_{t/2\pi}\xi(t)U_{-t/2\pi})=-\Lambda_{\mu}\Phi(\xi)(t)+\Phi(\xi')(t)+\Phi(\xi)(t)\Lambda_{\mu}=\Phi(\xi')(t)-\overline{D}_{\mu}(\Phi(\xi))(t).
\end{aligned}
\end{equation*}
On the other hand, since $\phi_t$ ($t\in\RR$) fixes $\ttt_0$ pointwise, it preserves all $\ttt_0$-weight spaces in $\kk_{\CC}$, and thus commutes with $\overline{D}_{\nu}$. 
Hence
\begin{equation*}
\begin{aligned}
[\widehat{\Phi}((0,0,1)),\widehat{\Phi}(\xi)](t)&=[(0,0,1),\Phi(\xi)](t)=D_{\mu+\nu}(\Phi(\xi))(t)\\
&=(\Phi(\xi))'(t)+\overline{D}_{\mu}(\Phi(\xi))(t)+\overline{D}_{\nu}(\Phi(\xi))(t)=\Phi(\xi')(t)+\overline{D}_{\nu}(\phi_{t/2\pi}(\xi(t)))\\
&=\Phi(\xi')(t)+\phi_{t/2\pi}(\overline{D}_{\nu}(\xi(t)))=\widehat{\Phi}(D_{\nu}(\xi))(t)=\widehat{\Phi}([(0,0,1),\xi])(t),
\end{aligned}
\end{equation*}
so that $\widehat{\Phi}$ is indeed a Lie algebra isomorphism. Moreover, the restriction of $\widehat{\Phi}$ to $\ttt_0^e$ is the identity because $\phi_t$ fixes $\ttt_0$ pointwise. This proves (i).

As $\widehat{\Phi}$ fixes $\ttt_0^e$, it preserves the root space decompositions of $\g_{\varphi}$ and $\g_{\psi}$ with respect to $\ttt_0^e$ and hence induces an isomorphism of locally affine root systems
$$\pi\co (\Delta_{\g_{\varphi}})_c\to (\Delta_{\g_{\psi}})_c:(\alpha,n)\mapsto (\alpha,\overline{\pi}_{\alpha}(n))$$
such that
$$(\alpha,n)(z,h,t)=\alpha(h)+it\big(\tfrac{n}{N_{\varphi}}+\nu(\alpha^{\sharp})\big)=\alpha(h)+it\big(\tfrac{\overline{\pi}_{\alpha}(n)}{N_{\psi}}+(\mu+\nu)(\alpha^{\sharp})\big)=(\alpha,\overline{\pi}_{\alpha}(n))(z,h,t)$$
for all $(z,h,t)\in\ttt_0^e$. This yields in particular
\begin{equation*}
\overline{\pi}_{\alpha}(n)=N_{\psi}\cdot(\tfrac{n}{N_{\varphi}}-\mu(\alpha^{\sharp})),
\end{equation*}
so that (ii) holds.

Since moreover
$$(\alpha,n)^{\vee}=\bigg(\frac{-2i\big(\tfrac{n}{N_{\varphi}}+\nu(\alpha^{\sharp})\big)}{(\alpha,\alpha)},\check{\alpha},0\bigg)=\bigg(\frac{-2i\big(\tfrac{\overline{\pi}_{\alpha}(n)}{N_{\psi}}+(\mu+\nu)(\alpha^{\sharp})\big)}{(\alpha,\alpha)},\check{\alpha},0\bigg)=(\alpha,\overline{\pi}_{\alpha}(n))^{\vee}$$
for all $(\alpha,n)\in (\Delta_{\g_{\varphi}})_c$ by (\ref{eqn:coroot}), we deduce that
$$r_{(\alpha,n)}=r_{(\alpha,\overline{\pi}_{\alpha}(n))}\in\GL(\ttt_0^e) \quad\textrm{for all $(\alpha,n)\in (\Delta_{\g_{\varphi}})_c$.}$$
Hence $\WW(\g_{\varphi},\ttt_0^e)=\WW(\g_{\psi},\ttt_0^e)$, proving (iii). 
\end{proof}

\begin{remark}\label{remark:conventions}
Recall from \S\ref{subsection:LA} that we defined the $\varphi$-twisted loop algebra $\LLL_{\varphi}(\kk)$ as a subspace of the $2\pi N_{\varphi}$-periodic functions of $C^{\infty}(\RR,\kk)$, where $N_{\varphi}$ is the order of $\varphi$. Another convention which one finds in the literature is to consider $2\pi$-periodic functions instead. We now explain how the characteristic 
data of these two definitions are related.

For $N\in\NN$, we set 
$$\LLL_{\varphi,N}(\kk):=\big\{\xi\in C^{\infty}(\RR,\kk) \ | \ \xi(t+\tfrac{2\pi}{N})=\varphi\inv(\xi(t))\big\}.$$
For $N=1$, this is the $\varphi$-twisted loop algebra $\LLL_{\varphi}(\kk)$ that we consider in this paper; the other convention which we alluded to above is to take $N=N_{\varphi}$.

Let $\ttt_0=\ttt^{\varphi}$ be maximal abelian in $\kk^{\varphi}$ and let $\nu\in i\ttt_0^*$. Define the skew-symmetric derivation $D_{\nu}$ of $\LLL_{\varphi,N}(\kk)$ as in \S\ref{subsection:D}. Thus $$D_{\nu}(e^{intN/N_{\varphi}}\otimes x)=i\big(\tfrac{nN}{N_{\varphi}}+\nu(\alpha^{\sharp})\big)(e^{intN/N_{\varphi}}\otimes x)$$
for all $\alpha\in\Delta_{\varphi}$ and $x\in\kk_{\CC}^{(\alpha,n)}$. Denote by 
$$\widehat{\LLL}_{\varphi,N}^{\nu}(\kk):=(\RR\oplus_{\omega_{D_{\nu}}}\LLL_{\varphi,N}(\kk))\rtimes_{\widetilde{D}_{\nu}}\RR$$
the double extension of $\LLL_{\varphi,N}(\kk)$ corresponding to $D_{\nu}$ as in \S\ref{subsection:DE}, with Cartan subalgebra $\ttt_0^e:=\RR\oplus\ttt_0\oplus\RR$. 
Then 
\begin{equation}
\Phi\co\widehat{\LLL}^{\nu}_{\varphi,1}(\kk)\to \widehat{\LLL}^{N\nu}_{\varphi,N}(\kk): \xi(t)\mapsto \xi(Nt), \quad (1,0,0)\mapsto (N,0,0), \quad (0,0,1)\mapsto (0,0,\tfrac{1}{N})
\end{equation}
is an isomorphism. 
Set
$$\bc:=(i,0,0)\in i\ttt_0^e\quad\textrm{and}\quad \bd:=(0,0,-i)\in i\ttt_0^e.$$ 
For a weight $\la=[\la_c,\la^0,\la_d]\in i(\ttt_0^e)^*$ with respect to $\ttt_0^e\subseteq \widehat{\LLL}^{\nu}_{\varphi,1}(\kk)$, where $$\la^0:=\la|_{(\ttt_0)_{\CC}}\in i\ttt_0^*,\quad \la_c:=\la(\bc)\in\RR\quad\textrm{and}\quad\la_d:=\la(\bd)\in\RR,$$
the corresponding weight with respect to $\ttt_0^e\subseteq \widehat{\LLL}^{\nu}_{\varphi,N}(\kk)$ is then given by 
\begin{equation}\label{eqn:laN}
\la\circ\Phi\inv=[\tfrac{1}{N}\la_c,\la^0,N\la_d].
\end{equation}
Similarly, for $\chi=[\chi_c,\chi^0,\chi_d]:=\chi_c\bc+\chi^0+\chi_d\bd\in i\ttt_0^e\subseteq \widehat{\LLL}^{\nu}_{\varphi,1}(\kk)$, where $\chi^0\in i\ttt_0$ and $\chi_c,\chi_d\in\RR$, we have
\begin{equation}\label{eqn:chiN}
\Phi(\chi)=[N\chi_c,\chi^0,\tfrac{1}{N}\chi_d].
\end{equation}
Using the identities (\ref{eqn:laN}) and (\ref{eqn:chiN}), it is then easy to state 
the positive energy condition for highest weight representations of $\widehat{\LLL}^{N\nu}_{\varphi,N}(\kk)$ in terms of the corresponding condition for $\widehat{\LLL}^{\nu}_{\varphi,1}(\kk)$ (see \S\ref{subsection:ATPER} for more detail about the positive energy condition).
\end{remark}

\section{The structure of finite order antiunitary operators}
Given a complex Hilbert space $\HH$ with orthonormal basis $\mathcal B=\{e_j \ | \ j\in J\}$, we denote by $\sigma_{\mathcal B}$ the complex conjugation on $\HH$ with respect to this basis. The following proposition describes the structure of finite order antiunitary operators on $\HH$.

\begin{prop}\label{prop:structure_antiunitary}
Let $\HH$ be a complex Hilbert space, and let $A$ be an antiunitary operator on $\HH$ of finite order. Let $N\in\NN$ be such that $A^{2N}=\id_{\HH}$ and set $\zeta:=e^{i\pi/N}\in\CC$. Then the following holds:
\begin{enumerate}
\item[(i)]
$A^2\in U(\HH)$ and $\HH$ has a decomposition $$\HH=\HH_1\oplus \HH_{-1} \oplus \bigoplus_{0<n<N/2}{(\HH_{\zeta^{2n}}\oplus \HH_{\zeta^{-2n}})}$$
into $A^2$-eigenspaces, where $\HH_{\la}$ denotes the $A^2$-eigenspace corresponding to the eigenvalue $\la$.
\item[(ii)]
There is some $A$-stable subspace $\widetilde{\HH}_1$ of $\HH_1$ with $\dim\widetilde{\HH}_1\leq 1$ such that the following holds: 
\begin{itemize}
\item
If $\dim\widetilde{\HH}_1=1$, there is some unit vector $e_{j_0}\in \widetilde{\HH}_1$ such that $Ae_{j_0}=e_{j_0}$.
\item
For each $n\in\ZZ$ with $0 \leq n\leq N/2$, there exists an orthonormal basis 
$\{e_j^+ ,e_j^-\ | \ j\in J_{\zeta^{2n}}\}$ of 
$\HH_{\zeta^{2n}}+\HH_{\zeta^{-2n}}$ (for $n > 0$) and of 
$\HH_1\cap \widetilde{\HH}_1^{\perp}$ (for $n = 0$) such that $A$ stabilises each plane $\CC e_j^+\oplus\CC e_j^-$, $j\in J_{\zeta^{2n}}$, and has the form
$$\begin{pmatrix}0 & \zeta^n \\ \zeta^{-n} & 0\end{pmatrix}\sigma_{e_j^+,e_j^-}$$
in the basis $\{e_j^+,e_j^-\}$.
\end{itemize}
\end{enumerate}
\end{prop}
\begin{proof}
The first statement of the proposition is clear. Since $A$ is antilinear and commutes with $A^2$, we have for any eigenvalue $\la$ of $A^2$ and any $v\in\HH_{\la}$ that 
$$A^2(Av)=A(A^2v)=A(\la v)=\overline{\la}Av,$$
and hence $A.\HH_{\la}=\HH_{\overline{\la}}$.
Thus $A$ stabilises each of the subspaces $\HH_1$, $\HH_{-1}$ and 
$\HH_{\zeta^{2n}}\oplus \HH_{\zeta^{-2n}}$ for $n\in\ZZ$ with $0<n<N/2$.

Since $A$ acts as a conjugation on $\HH_1$, the fixed-point space $\HH_1^A$ is a real form of $\HH_1$. Choose an orthornormal basis $\{f_j^+,f_j^- \ | \ j\in J_{1}\}\cup S_1$ of $\HH_1$ that is contained in $\HH_1^A$, where $S_1=\{e_{j_0}\}$ is a singleton if $\dim\HH_1$ is finite and odd and $S_1=\varnothing$ otherwise. Let also $\widetilde{\HH}_1$ denote the sub-vector space of $\HH_1$ with basis $S_1$. Thus $Af_j^{\pm}=f_j^{\pm}$ for all $j\in J_1$ and $Ae_{j_0}=e_{j_0}$ if $\dim\widetilde{\HH}_1=1$.
For all $j\in J_1$, set $$e_j^{\pm}:=\frac{1}{\sqrt{2}}(f_j^+\pm if_j^-).$$ 
Then $\{e_j^+,e_j^- \ | \ j\in J_{1}\}$ is an orthonormal basis of $\HH_1\cap \widetilde{\HH}_1^{\perp}$, and $A$ has the desired form in each of the bases $\{e_j^+,e_j^-\}$, $j\in J_{1}$.

Let now $n\in\ZZ$ with $0<n<N/2$, and choose some orthonormal basis $(e_j^+)_{j\in J_{\zeta^{2n}}}$ of $\HH_{\zeta^{2n}}$. Set $e_j^-:=\zeta^n Ae_j^+\in\HH_{\zeta^{-2n}}$ for each $j\in J_{\zeta^{2n}}$. Then $\{e_j^+,e_j^- \ | \ j\in J_{\zeta^{2n}}\}$ is an orthonormal basis of $\HH_{\zeta^{2n}}\oplus \HH_{\zeta^{-2n}}$, and $A$ has the desired form in each of the bases $\{e_j^+,e_j^-\}$, $j\in J_{\zeta^{2n}}$.

Finally, note that for any unit vector $v\in\HH_{-1}$, the subspace $\CC v\oplus\CC Av$ is two-dimensional and stabilised by $A$. Indeed, 
$$\langle v,Av\rangle=\overline{\langle Av,A^2v\rangle}=-\overline{\langle Av,v\rangle}=-\langle v,Av\rangle$$
and hence $\langle v,Av\rangle=0$, so that $\{v,Av\}$ is an orthonormal basis of $\CC v\oplus\CC Av$.
Using Zorn's lemma, we may thus choose an orthonormal subset $(e_j^+)_{j\in J_{-1}}$ in $\HH_{-1}$ such that $\{e_j^+,e_j^-:=\zeta^{N/2}Ae_j^+ \ | \ j\in J_{-1}\}$ is an orthonormal basis of $\HH_{-1}$. Again, $A$ has the desired form in each of the bases $\{e_j^+,e_j^-\}$, $j\in J_{-1}$.
This concludes the proof of the proposition. 
\end{proof}

\section{Root data for affinisations of Hilbert--Lie algebras}\label{section:root_data}
Let $\kk=\uu_2(\HH_{\KK})$ for some infinite-dimensional Hilbert space $\HH_{\KK}$ over $\KK=\RR$, $\CC$ or $\HHH$, and let $\varphi\in\Aut(\kk)$ be of finite order.

\begin{lemma}\label{lemma:choiceA}
There exists some unitary (if $\KK=\RR,\CC,\HHH$) or antiunitary (if $\KK=\CC$) operator $A$ on $\HH_{\KK}$ of finite order such that $\varphi=\pi_A$. 
\end{lemma}
\begin{proof}
By \S\ref{subsection:AG}, there exists some unitary (if $\KK=\RR,\CC,\HHH$) or antiunitary (if $\KK=\CC$) operator $B$ on $\HH_{\KK}$ with $\varphi=\pi_B$. Let $N\in\NN$ be the order of $\varphi$. Then $B^N$ centralises $\kk$, and hence $B^N=\la_0\cdot\id_{\HH_{\KK}}$ for some $\la_0$ in the center of $\KK$ with $|\la_0|=1$.

If $\KK=\RR$ or $\KK=\HHH$, then $\la_0\in\{\pm 1\}$ and hence $A:=B$ satisfies $A^{2N}=\id$.
If $\KK=\CC$ and $B$ is unitary, we set $A:=\nu B$ for some $N$-th root $\nu\in\CC$ of $\la_0\inv$, so that $A^N=\id$.
Finally, assume that $\KK=\CC$ and that $B$ is antiunitary. Since for any nonzero $v\in\HH_{\CC}$,
$$\la_0(Bv)=B^N(Bv)=B(B^Nv)=B(\la_0v)=\overline{\la}_0(Bv),$$
we get $\la_0=\overline{\la}_0\in\RR$ and hence $\la_0\in\{\pm 1\}$. We then set $A:=B$, so that $A^{2N}=\id$.
\end{proof}

For each pair $(\kk,\varphi)$, we now describe data $\psi$, $(\phi_t)_{t\in\RR}$, $\ttt$ and $\mu$ as in Proposition~\ref{prop:untwisting_extension2}, thus yielding the desired isomorphisms from Theorem~\ref{thm:mainintro}.

\begin{example}\label{ex:CL}
Let $\kk=\uu_2(\HH)$ for some infinite-dimensional complex Hilbert space $\HH$ and let $\varphi\in\Aut(\kk)$ be of finite order. Let $A$ be some unitary operator on $\HH$ of finite order $N\in\NN$ such that $\varphi=\pi_A$ (see Lemma~\ref{lemma:choiceA}). Set $\zeta:=e^{2i\pi/N}$.

As every unitary representation of the cyclic group of order $N$ on $\HH$ is a direct sum of $1$-dimensional irreducible ones, we may choose some orthonormal basis $(e_j)_{j\in J}$ of $\HH$ consisting of $A$-eigenvectors. Let $\ttt\subseteq\kk$ be the subalgebra of all diagonal operators with respect to the $e_j$, $j\in J$. Then $\ttt$ is elliptic and maximal abelian, and $\ttt_{\CC}\cong \ell^2(J,\CC)$ with respect to the orthonormal basis $\{E_j \ | \ j\in J\}\subseteq i\ttt$ given by $E_je_k:=\delta_{jk}e_k$. The set of roots of $\kk_{\CC}\cong\gl_2(\HH)$ with respect to $\ttt_{\CC}$ is given by the root system
$$A_J=\{\epsilon_j-\epsilon_k \ | \ j\neq k, \ j,k\in J\},$$
where $\epsilon_j(E_k):=\delta_{jk}$. 
The corresponding set of compact roots for $\widehat{\LLL}(\kk)$ is of type $A_J^{(1)}$ (see \cite[Examples~1.10 and 2.4]{Hloopgroups}).

For each $j\in J$, let $n_j\in\{0,1,\dots,N-1\}$ be such that $Ae_j=\zeta^{n_j}e_j$. 
For each $t\in\RR$, let also $U_t\in U(\HH)$ be the diagonal operator defined by $U_te_j=\zeta^{tn_j}e_j$, and set $\phi_t=\pi_{U_t}\in\Aut(\kk)$. Let $\mu\in i\ttt^*$ be defined by $\mu(E_{j}):=\mu_j:=-n_j/N$. 
Setting $\Lambda_{\mu}:=i\mu^{\sharp}=\sum_{j\in J}{i\mu_jE_{j}}\in\uu(\HH)$, we then have
$$\frac{d}{dt}U_{t/2\pi}=-\Lambda_{\mu}U_{t/2\pi}=-U_{t/2\pi}\Lambda_{\mu}.$$
Finally, note that, for any $t\in\RR$ and $x\in\ttt$, the operators $U_t$ and $x$ are both diagonal with respect to the $e_j$, $j\in J$, and hence commute. In particular $\ttt_0:=\ttt^{\varphi}=\ttt\subseteq\ttt^{\phi_t}$ for all $t\in\RR$.

We may thus apply Proposition~\ref{prop:untwisting_extension2} (with $\psi=\id$) and conclude that for any $\nu\in i\ttt_0^*$, there is an isomorphism
$$\widehat{\Phi}\co \widehat{\LLL}_{\varphi}^{\nu}(\kk)\to \widehat{\LLL}^{\mu+\nu}(\kk):(z_1,\xi(t),z_2)\mapsto (z_1,\phi_{t/2\pi}(\xi(t)),z_2)$$
fixing the Cartan subalgebra $\ttt_0^e:=\RR\oplus\ttt_0\oplus\RR$ pointwise. Moreover, $\widehat{\Phi}$ induces an isomorphism of locally affine root systems $$\pi\co \Delta(\widehat{\LLL}_{\varphi}^{\nu}(\kk),\ttt_0^e)_c\to A_J^{(1)}=\Delta(\widehat{\LLL}^{\mu+\nu}(\kk),\ttt_0^e)_c:(\alpha,n)\mapsto \big(\alpha,\tfrac{n}{N_{\varphi}}-\mu(\alpha^{\sharp})\big),$$
where $N_{\varphi}$ is the order of $\varphi$. Finally, the Weyl groups $\WW(\widehat{\LLL}_{\varphi}^{\nu}(\kk),\ttt_0^e)$ and $\WW(\widehat{\LLL}^{\mu+\nu}(\kk),\ttt_0^e)$ coincide.
\end{example}

\begin{example}\label{ex:HL}
Let $\kk=\uu_2(\HH_{\HHH})$ for some infinite-dimensional quaternionic Hilbert space $\HH_{\HHH}$ and let $\varphi\in\Aut(\kk)$ be of finite order. Let $A$ be some unitary operator on $\HH_{\HHH}$ of finite order $N\in\NN$ such that $\varphi=\pi_A$ (see Lemma~\ref{lemma:choiceA}). Set $\zeta:=e^{2i\pi/N}$.

The quaternionic Hilbert space $\HH_{\HHH}$ can be constructed as $\HH_{\HHH}=\HH^2$ for some complex Hilbert space $\HH$ with conjugation $\sigma$, where the quaternionic structure on $\HH^2$ is defined by the antilinear isometry\footnote{Writing $\CC=\RR+\RR\mathcal I$ and $\HHH=\CC+\CC\mathcal J=\RR+\RR\mathcal I+\RR\mathcal J+\RR\mathcal I \mathcal J$, the isometry $\widetilde{\sigma}$ corresponds to left multiplication by~$\mathcal J$.} $\widetilde{\sigma}(v,w):=(-\sigma w,\sigma v)$. With this identification, we then have
$$\kk=\uu_2(\HH_{\HHH})=\{x\in\uu_2(\HH^2) \ | \ \widetilde{\sigma}x=x\widetilde{\sigma}\}$$
and
$$U(\HH_{\HHH})=\{g\in U(\HH^2) \ | \ \widetilde{\sigma}g\widetilde{\sigma}\inv=g\}.$$
Let $$\HH^2=(\HH^2)_1\oplus (\HH^2)_{-1} \oplus \bigoplus_{0<n<N/2}{((\HH^2)_{\zeta^n}\oplus (\HH^2)_{\zeta^{-n}})}$$ be the decomposition of $\HH^2$ into $A$-eigenspaces, where $(\HH^2)_{\la}$ denotes the $A$-eigenspace corresponding to the eigenvalue $\la$. Since $A\in U(\HH^2)$ commutes with $\widetilde{\sigma}$, the antilinear isometry $\widetilde{\sigma}$ maps $(\HH^2)_{\la}$ to $(\HH^2)_{\overline{\la}}$ for each $A$-eigenvalue $\la$. In particular, each of the subspaces $(\HH^2)_1$, $(\HH^2)_{-1}$ and $(\HH^2)_{\zeta^n}\oplus (\HH^2)_{\zeta^{-n}}$ with $0<n<N/2$ is a quaternionic Hilbert subspace of $\HH_{\HHH}$.
Let $(e_j)_{j\in J_{\pm 1}}$ be an orthonormal basis of $(\HH^2)_{\pm 1}$ over $\HHH$, and for each $n$ with $0<n<N/2$, let $(e_j)_{j\in J_{\zeta^n}}$ be an orthonormal basis of $(\HH^2)_{\zeta^n}\oplus (\HH^2)_{\zeta^{-n}}$ over $\HHH$ that is contained in $(\HH^2)_{\zeta^n}$. Then the reunion of these bases yields an orthonormal basis $(e_j)_{j\in J}$ of $\HH_{\HHH}$ over $\HHH$. 

Up to replacing $\HH$ by the complex Hilbert space generated by $(e_j)_{j\in J}$, we may then assume that $\HH_{\HHH}$ is constructed as above as $\HH_{\HHH}=\HH^2$ for some complex Hilbert space $\HH$ with orthonormal basis $(e_j)_{j\in J}$ (and complex conjugation $\sigma$ with respect to this basis) in such a way that for each $j\in J$, the $\CC$-basis vectors $(e_j,0)$ and $(0,e_j)$ are $A$-eigenvectors of respective eigenvalues $\zeta^{n_j}$ and $\zeta^{-n_j}$ for some natural number $n_j$ with $0\leq n_j\leq N/2$.

Let $\ttt\subseteq\kk$ be the subalgebra of all diagonal operators with respect to the $(e_j,0)$ and $(0,e_j)$, $j\in J$. 
Then $\ttt$ is elliptic and maximal abelian, and $\ttt_{\CC}$ consists of diagonal operators in 
$$\kk_{\CC}=\uu_2(\HH_{\HHH})_{\CC}=\big\{(\begin{smallmatrix}A & B \\ C & -A^T\end{smallmatrix})\in B_2(\HH^2) \ | \ B^T=B, \ C^T=C\big\}$$
of the form $h=\diag((h_j),(-h_j))$, where $B^T:=\sigma B^*\sigma$ for all $B\in B(\HH)$. Thus $\ttt_{\CC}\cong \ell^2(J,\CC)$ with respect to the orthonormal basis $\{E_j \ | \ j\in J\}\subseteq i\ttt$ defined by $E_j(e_k,0):=\delta_{jk}(e_k,0)$ and $E_j(0,e_k):=-\delta_{jk}(0,e_k)$.
The set of roots of $\kk_{\CC}$ with respect to $\ttt_{\CC}$ is given by the root system
$$C_J=\{\pm 2\epsilon_j, \pm(\epsilon_j\pm\epsilon_k) \ | \ j\neq k, \ j,k\in J\},$$ where $\epsilon_j(E_k):=\delta_{jk}$. 
The corresponding set of compact roots for $\widehat{\LLL}(\kk)$ is of type $C_J^{(1)}$ (see \cite[Examples~1.12 and 2.4]{Hloopgroups}).

For each $t\in\RR$, let $U_t\in U(\HH_{\HHH})$ be the diagonal operator defined by $U_t(e_j,0)=\zeta^{tn_j}(e_j,0)$ (and hence $U_t(0,e_j)=\zeta^{-tn_j}(0,e_j)$) and set $\phi_t=\pi_{U_t}\in\Aut(\kk)$.
Let $\mu\in i\ttt_{\CC}^*$ be defined by $\mu(E_{j}):=\mu_j:=-n_j/N$. Setting $\Lambda_{\mu}:=i\mu^{\sharp}=\sum_{j\in J}{i\mu_jE_{j}}\in\uu(\HH_{\HHH})$, we then have
$$\frac{d}{dt}U_{t/2\pi}=-\Lambda_{\mu}U_{t/2\pi}=-U_{t/2\pi}\Lambda_{\mu}.$$
Moreover, for any $t\in\RR$ and $x\in\ttt$, the operators $U_t$ and $x$ are both diagonal with respect to the $(e_j,0)$ and $(0,e_j)$, $j\in J$, and hence commute. In particular $\ttt_0:=\ttt^{\varphi}=\ttt\subseteq\ttt^{\phi_t}$ for all $t\in\RR$.

We may thus apply Proposition~\ref{prop:untwisting_extension2} (with $\psi=\id$) and conclude that for any $\nu\in i\ttt_0^*$, there is an isomorphism
$$\widehat{\Phi}\co \widehat{\LLL}_{\varphi}^{\nu}(\kk)\to \widehat{\LLL}^{\mu+\nu}(\kk):(z_1,\xi(t),z_2)\mapsto (z_1,\phi_{t/2\pi}(\xi(t)),z_2)$$
fixing the Cartan subalgebra $\ttt_0^e:=\RR\oplus\ttt_0\oplus\RR$ pointwise. Moreover, $\widehat{\Phi}$ induces an isomorphism of locally affine root systems $$\pi\co \Delta(\widehat{\LLL}_{\varphi}^{\nu}(\kk),\ttt_0^e)_c\to C_J^{(1)}=\Delta(\widehat{\LLL}^{\mu+\nu}(\kk),\ttt_0^e)_c:(\alpha,n)\mapsto \big(\alpha,\tfrac{n}{N_{\varphi}}-\mu(\alpha^{\sharp})\big),$$
where $N_{\varphi}$ is the order of $\varphi$. Finally, the Weyl groups $\WW(\widehat{\LLL}_{\varphi}^{\nu}(\kk),\ttt_0^e)$ and $\WW(\widehat{\LLL}^{\mu+\nu}(\kk),\ttt_0^e)$ coincide.
\end{example}

\begin{example}\label{ex:RL}
Let $\kk=\uu_2(\HH_{\RR})$ for some infinite-dimensional real Hilbert space $\HH_{\RR}$ and let $\varphi\in\Aut(\kk)$ be of finite order. Let $A$ be some unitary operator on $\HH_{\RR}$ of finite order $N\in\NN$ such that $\varphi=\pi_A$ (see Lemma~\ref{lemma:choiceA}).

Note that $A$ may be viewed as a unitary operator on the complexification $\HH:=(\HH_{\RR})_{\CC}$ of $\HH_{\RR}$ that commutes with complex conjugation. Since moreover $\HH$ decomposes as an orthogonal direct sum of one-dimensional $A$-eigenspaces, $\HH_{\RR}$ decomposes as an orthogonal direct sum
$$\HH_{\RR}=\HH_1\oplus\HH_{-1}\oplus \HH_{\zeta},$$ 
where $\HH_{\pm 1}$ is the $A$-eigenspace for the eigenvalue $\pm 1$, and where $\HH_{\zeta}$ has an orthonormal basis $\{e_j,e'_j \ | \ j\in J_{\zeta}\}$ such that 
$A$ stabilises each plane $\RR e_j+\RR e'_j$ and is of the form
\begin{equation}\label{eqn:cossin}
\begin{pmatrix}\cos(2\pi n_j/N) & -\sin(2\pi n_j/N) \\ \sin(2\pi n_j/N) & \cos(2\pi n_j/N)\end{pmatrix}
\end{equation}
in the basis $\{e_j,e'_j\}$, for some $n_j\in\ZZ$ with $0<n_j<N/2$. We let also $\{e_j,e_j' \ | \ j\in J_{\pm 1}\}\cup S_{\pm 1}$ denote an orthonormal basis of $\HH_{\pm 1}$, where $S_{\pm 1}:=\{e_{j_{\pm 1}}\}$ is a singleton if $\dim\HH_{\pm 1}$ is finite and odd and $S_{\pm 1}:=\varnothing$ otherwise. Up to replacing $A$ by $-A$ (this does not modify $\varphi$), we may then assume that $S_1=\{e_{j_1}\}$ and $S_{-1}=\varnothing$ in case $|S_1\cup S_{-1}|=1$. Note that $N$ must be even if $\HH_{-1}\neq\{0\}$.

Set $n_j:=0$ (resp. $n_j:=N/2$) for each $j\in J_1$ (resp. $j\in J_{-1}$). Writing $J':=J_{\zeta}\cup J_1\cup J_{-1}$, we thus get an orthonormal decomposition
$$\HH_{\RR}=\RR e_{j_1}\oplus\RR e_{j_{-1}}\oplus\widehat{\bigoplus}_{j\in J'}{(\RR e_j\oplus\RR e'_j)}$$
such that $A$ stabilises each plane $\RR e_j+\RR e_j'$ ($j\in J'$) and is of the form (\ref{eqn:cossin}) in the basis $\{e_j,e'_j\}$, and with the convention that $e_{j_{\pm 1}}\in\HH_{\pm 1}$ is omitted if $S_{\pm 1}=\varnothing$.
If $|S_1\cup S_{-1}|=2$, we set $e'_{j_1}:=e_{j_{-1}}$ and $J:=J'\cup\{j_1\}$. Otherwise, we set $J:=J'$. 

We choose a maximal abelian subalgebra $\ttt\subseteq\kk$ such that $\ker\ttt:=\{x\in\HH_{\RR} \ | \ h.x=0 \ \forall h\in\ttt\}$ is the one-dimensional subspace $\RR e_{j_1}$ if $|S_1\cup S_{-1}|=1$ and $\ker\ttt=\{0\}$ otherwise, such that $\ttt$ commutes with the orthogonal complex structure $\III$ on $(\ker\ttt)^{\perp}=\widehat{\bigoplus}_{j\in J}{(\RR e_j\oplus\RR e'_j)}$ defined by $\III e_j:=e'_j$ for all $j\in J$, and such that all planes $\RR e_j+\RR \III e_j$ ($j\in J$) are $\ttt$-invariant (see \cite[Example~1.13]{Hloopgroups}).
For $j\in J$, we define the elements
$$f_j:= \frac{1}{\sqrt{2}}(e_j-i\III e_j)\quad\textrm{and}\quad f_{-j}:= \frac{1}{\sqrt{2}}(e_j+i\III e_j)$$ of $\HH$.
If $\ker\ttt\neq \{0\}$, we also set $f_{j_1}:=e_{j_1}$. Then the $f_{j}$ form an orthonormal basis of $\HH$ consisting of $\ttt$-eigenvector. 

Note that for each $j\in J'$, the basis elements $f_j$ and $f_{-j}$ are also $A$-eigenvectors, with respective eigenvalues $\zeta^{n_j}$ and $\zeta^{-n_j}$. For each $t\in\RR$, we let $U_t\in U(\HH_{\RR})$ be defined by the matrix 
$$\begin{pmatrix}\cos(2t\pi n_j/N) & -\sin(2t\pi n_j/N) \\ \sin(2t\pi n_j/N) & \cos(2t\pi n_j/N)\end{pmatrix}$$
in the basis $\{e_j,e'_j\}$ for each $j\in J$, where we have set $n_{j_1}:=0$ in case $|S_1\cup S_{-1}|=2$, and by $U_te_{j_1}=e_{j_1}$ in case $|S_1\cup S_{-1}|=1$. We also set $\phi_t=\pi_{U_t}\in\Aut(\kk)$. Note then that $\phi_1\psi=\varphi$, where $\psi\in\Aut(\kk)$ is the order $1$ or $2$ automorphism $\psi=\pi_{B}$ of $\kk$ corresponding to the matrix $B\in U(\HH_{\RR})$ whose restriction to $\bigoplus_{j\in J'}(\RR e_j\oplus\RR e'_j)$ is the identity, and such that $Be_{j_{\pm 1}}=\pm e_{j_{\pm 1}}$ if $S_{\pm 1}\neq \varnothing$. Thus $\psi=\id$ unless $|S_1\cup S_{-1}|=2$. Note also that $\psi$ commutes with $\phi_t$ for all $t\in\RR$. 

The restriction of any $x\in\ttt$ to $\RR e_j\oplus\RR e_j'$, $j\in J$, is of the form $$\begin{pmatrix}0 & a \\ -a & 0\end{pmatrix}$$
in the basis $\{e_j,e_j'\}$, for some $a\in\RR$. In particular, $x$ commutes with $U_t$ for each $t\in\RR$, so that $\ttt^{\varphi}\subseteq\ttt=\ttt^{\phi_t}$ for all $t\in\RR$. 
The same argument implies that $\ttt^{\varphi}$ and $\ttt^{\psi}$ both contain the subspace
$$\ttt_0:=\{x\in\ttt \ | \ xe_j=xe_j'=0 \ \forall j\in J\setminus J'\}.$$

We now claim that $\ttt^{\varphi}$ and $\ttt^{\psi}$ coincide with $\ttt_0$ and are maximal abelian in $\kk^{\varphi}$ and $\kk^{\psi}$, respectively. Indeed, if $\psi=\id$, so that $\varphi=\phi_1$, we have $\ttt^{\varphi}=\ttt=\ttt^{\psi}=\ttt_0$. 
Assume now that $\psi=\pi_B$ has order $2$, so that $|S_1\cup S_{-1}|=2$. Then $\ker\ttt=\{0\}$ and $\ttt^{\varphi}=\ttt_0=\ttt^{\psi}$ because the restriction of $\varphi$ (resp. $\psi$) to $\RR e_{j_1}\oplus\RR e_{j_{-1}}$ is of the form $(\begin{smallmatrix}1 &0 \\ 0& -1\end{smallmatrix})$ in the basis $\{e_{j_1},e_{j_{-1}}\}$ and
$$\begin{pmatrix}1 & 0 \\ 0 & -1\end{pmatrix}\begin{pmatrix}0 & a \\ -a & 0\end{pmatrix}\begin{pmatrix}1 & 0 \\ 0 & -1\end{pmatrix}=\begin{pmatrix}0 & -a \\ a & 0\end{pmatrix}.$$
Let $x\in\kk^{\varphi}$ (resp. $\kk^{\psi}$) be such that $\ttt_0+\RR x$ is abelian.
Then $x$ stabilises each plane $\RR e_j+\RR e_j'$ ($j\in J'$) and hence decomposes as
$x=x_0+x_1$ for some $x_0\in\ttt_0$ and some $x_1\in\kk$ with $x_1e_j=xe_j'=0$ for all $j\in J'$. Since $x_1=x-x_0$ is skew-symmetric, it stabilises $\RR e_{j_1}\oplus\RR e_{j_{-1}}$. Moreover, $x_1$ is fixed by $\varphi$ (resp. $\psi$), and hence we conclude as above that $x_1=0$. Thus $x\in\ttt_0$, and hence $\ttt_0$ is maximal abelian in $\kk^{\varphi}$ (resp. $\kk^{\psi}$), as desired.

The complexification $(\ttt_0)_{\CC}$ of $\ttt_0$ is precisely the set of all those elements in $\kk^{\psi}_{\CC}$ which are diagonal with respect to the orthonormal basis of $\HH$ consisting of the $f_j$. Thus $(\ttt_0)_{\CC}\cong\ell^2(J',\CC)$ with respect to the orthonormal basis $\{E_j \ | \ j\in J'\}\subseteq i\ttt_0$ defined by $E_jf_k:=\delta_{jk}f_k$ for all $k\in J'$.
Define also $\epsilon_j\in i\ttt_0^*$ by $\epsilon_j(E_k)=\delta_{jk}$, $k\in J'$. 

If $\psi=\id$, the set of roots of $\kk^{\psi}_{\CC}=\kk_{\CC}$ with respect to $\ttt^{\psi}_{\CC}=\ttt_{\CC}$ is given by
$$D_J=\{\pm(\epsilon_j\pm \epsilon_k) \ | \ j\neq k, \ j,k\in J\}\quad\textrm{if $\ker\ttt=\{0\}$}$$
and
$$B_J=\{\pm(\epsilon_j\pm \epsilon_k) \ | \ j\neq k, \ j,k\in J\}\cup\{\pm\epsilon_j \ | \ j\in J\}\quad\textrm{otherwise.}$$
The set of compact roots for $\widehat{\LLL}_{\psi}(\kk)=\widehat{\LLL}(\kk)$ is then respectively of type $D_J^{(1)}$ and $B_J^{(1)}$ (see \cite[Examples~1.13 and 2.4]{Hloopgroups}).

If $\psi=\pi_B$ has order $2$ (so that $|S_1\cup S_{-1}|=2$) then, up to replacing $B$ by $-B$ (which does not modify $\psi$), the operator $B$ is the orthogonal reflection in the hyperplane $e_{j_{-1}}^{\perp}$. Thus $\psi$ is the standard automorphism from \cite[Example~2.8]{Hloopgroups}. As observed above, $\ttt^{\psi}=\ttt_0$ is maximal abelian in
$$\kk^{\psi}=\{x\in\kk \ | \ xe_{j_{-1}}=0\}$$ 
and $\ker(\ttt^{\psi})\cap e_{j_{-1}}^{\perp}=\RR e_{j_1}$ is one-dimensional. The set of roots of $\kk^{\psi}_{\CC}$ with respect to $\ttt^{\psi}_{\CC}$ is then of type $B_{J'}$, while the set of compact roots for $\widehat{\LLL}_{\psi}(\kk)$ is of type $B_J^{(2)}$
(see \cite[Example~2.8]{Hloopgroups}).

Back to the general case ($\psi$ of order $1$ or $2$), let $\mu\in i\ttt_0^*$ be defined by $\mu(E_{j}):=\mu_j:=-n_j/N$ for all $j\in J'$. Setting $\Lambda_{\mu}:=i\mu^{\sharp}=\sum_{j\in J'}{i\mu_jE_{j}}\in\uu(\HH_{\RR})$, we then have
$$\frac{d}{dt}U_{t/2\pi}=-\Lambda_{\mu}U_{t/2\pi}=-U_{t/2\pi}\Lambda_{\mu}.$$

We may thus apply Proposition~\ref{prop:untwisting_extension2} and conclude that for any $\nu\in i\ttt_0^*$, there is an isomorphism
$$\widehat{\Phi}\co \widehat{\LLL}_{\varphi}^{\nu}(\kk)\to \widehat{\LLL}_{\psi}^{\mu+\nu}(\kk):(z_1,\xi(t),z_2)\mapsto (z_1,\phi_{t/2\pi}(\xi(t)),z_2)$$
fixing the Cartan subalgebra $\ttt_0^e:=\RR\oplus\ttt_0\oplus\RR$ pointwise. Moreover, $\widehat{\Phi}$ induces an isomorphism of locally affine root systems $$\pi\co \Delta(\widehat{\LLL}_{\varphi}^{\nu}(\kk),\ttt_0^e)_c\to \Delta(\widehat{\LLL}_{\psi}^{\mu+\nu}(\kk),\ttt_0^e)_c\in\{D_J^{(1)},B_J^{(1)},B_J^{(2)}\}:(\alpha,n)\mapsto \big(\alpha,N_{\psi}\cdot(\tfrac{n}{N_{\varphi}}-\mu(\alpha^{\sharp}))\big),$$
where $N_{\varphi}$ is the order of $\varphi$ and $N_{\psi}\in\{1,2\}$ is the order of $\psi$. Finally, the Weyl groups $\WW(\widehat{\LLL}_{\varphi}^{\nu}(\kk),\ttt_0^e)$ and $\WW(\widehat{\LLL}_{\psi}^{\mu+\nu}(\kk),\ttt_0^e)$ coincide.
\end{example}

\begin{example}\label{ex:CA}
Let $\kk=\uu_2(\HH)$ for some infinite-dimensional complex Hilbert space $\HH$ and let $\varphi\in\Aut(\kk)$ be of finite order. Let $A$ be some antiunitary operator on $\HH$ of finite order $2N$ for some $N\in\NN$ and such that $\varphi=\pi_A$ (see Lemma~\ref{lemma:choiceA}).

We set $\zeta:=e^{i\pi/N}$ and we apply Proposition~\ref{prop:structure_antiunitary} to $A$.
Let $\widetilde{\HH}_1$ and $\{e_j^+ ,e_j^-\ | \ j\in J_{\zeta^{2n}}\}$, $0\leq n\leq N/2$, be as in Proposition~\ref{prop:structure_antiunitary}(ii). Thus $A$ stabilises each plane $\CC e_j^+\oplus\CC e_j^-$, $j\in J_{\zeta^{2n}}$, and has the form
$$\begin{pmatrix}0 & \zeta^n \\ \zeta^{-n} & 0\end{pmatrix}\sigma_{e_j^+,e_j^-}$$
in the basis $\{e_j^+,e_j^-\}$. If $\dim\widetilde{\HH}_1=1$, we also choose some basis $S_1:=\{e_{j_0}\}$ of $\widetilde{\HH}_1$ such that $Ae_{j_0}=e_{j_0}$; if $\dim\widetilde{\HH}_1=0$, we set $S_1:=\varnothing$.
Set $$J:=\bigcup_{0\leq n\leq N/2}{J_{\zeta^{2n}}}\quad\textrm{and}\quad \BBB_{\pm}:=\{e_j^{\pm} \ | \ j\in J\}.$$
Then $\BBB:=\BBB_+\cup\BBB_-\cup\ S_1$ is an orthonormal basis of $\HH$. For each $j\in J$ and $n\in\ZZ$ with $0\leq n\leq N/2$, we set $n_j:=n$ if $j\in J_{\zeta^{2n}}$. 

To treat both cases $S_1=\varnothing$ and $S_1\neq\varnothing$ at once, we adopt the convention that whenever $e_{j_0}$ appears in what follows, it should be omitted if $S_1=\varnothing$. We also set $\epsilon:=i\in\CC$ if $S_1=\varnothing$ and $\epsilon:=1\in\CC$ otherwise.

Let $\HH_0^{\pm}$ be the closed subspace of $\HH$ spanned by $\BBB^{\pm}$. We define on  $\HH_0^{+}$ the complex conjugation $\sigma_0^{+}=\sigma_{\epsilon\BBB_{+}}$ with respect to $\epsilon\BBB_{+}=\{\epsilon e_j^{+} \ | \ j\in J\}$ and on $\HH_0^{-}$ the complex conjugation $\sigma_0^{-}=\sigma_{\BBB_{-}}$ with respect to $\BBB_{-}$. 
Write $\HH$ as $$\HH=\HH_0^+\oplus \CC e_{j_0} \oplus \HH_0^-,$$ which we endow with the conjugation $\sigma$ extending $\sigma_0^+$ and $\sigma_0^-$, and such that $\sigma e_{j_0}=e_{j_0}$.
Consider the automorphism $\psi\in \Aut(\kk)$ defined by
$$\psi(x)=S\sigma x (S\sigma)\inv\quad\textrm{for $x\in\kk$,}$$
where
$$S=\begin{pmatrix}0&0&\mathbf 1\\ 0 & 1 & 0\\ \mathbf 1 &0 &0\end{pmatrix}\quad\textrm{if $S_1\neq\varnothing$}\quad\textrm{and}\quad S=\begin{pmatrix}0&\mathbf 1\\ -\mathbf 1  &0\end{pmatrix}\quad\textrm{if $S_1=\varnothing$}.$$
Setting $\widetilde{\sigma}:=S\sigma$, we thus have 
$$\widetilde{\sigma}e_j^{\pm}=e_j^{\mp}\quad\textrm{for all $j\in J$}\quad\textrm{and}\quad\widetilde{\sigma}e_{j_0}=e_{j_0}.$$
Note that $\psi$ is the standard automorphism of \cite[Example~2.9]{Hloopgroups} if $S_1=\varnothing$ and of \cite[Example~2.10]{Hloopgroups} if $S_1\neq\varnothing$.

Let $(U_t)_{t\in\RR}$ be the one-parameter group of unitary operators on $\HH$ defined by
$$U_te_j^{\pm}=\zeta^{\pm n_jt}e_j^{\pm}\quad\textrm{for all $j\in J$}\quad\textrm{and}\quad U_te_{j_0}=e_{j_0}.$$
Then $A=U_1\widetilde{\sigma}$ and $U_t$ commutes with $\widetilde{\sigma}$ for each $t\in\RR$. Set $\phi_t:=\pi_{U_t}\in\Aut(\kk)$ for all $t\in\RR$. Thus $\phi_t$ commutes with $\psi$ and $\phi_1\psi=\varphi$.

Let $\ttt$ be the set of elements in $\kk$ which are diagonal with respect to the orthonormal basis $\BBB$. Hence $\ttt$ is maximal abelian in $\kk$ (see Example~\ref{ex:CL}) and $\ttt=\ttt^{\phi_t}$ for each $t\in\RR$. For each $j\in J$, define the operator $E_j\in i\ttt$ by $$E_je_k^{\pm}:=\pm\delta_{jk}e_k^{\pm}\quad\textrm{for all $k\in J$}\quad\textrm{and}\quad E_je_{j_0}:=0.$$
Let $$\ttt_0:=\sppan_{\RR}\{iE_j \ | \ j\in J\}\subseteq\ttt.$$ 
Since for each $j\in J$ and $t\in\RR$, the operator $iE_j$ commutes with $U_t$ and $\widetilde{\sigma}$, the subalgebras $\ttt^{\varphi}\subseteq\kk^{\varphi}$ and $\ttt^{\psi}\subseteq\kk^{\psi}$ both contain $\ttt_0$ and are contained in $\ttt^{\phi_t}$. On the other hand, if $x\in\kk$ centralises $\ttt_0$, then $x$ is diagonal with respect to $\BBB$, that is, $x\in\ttt$. In particular, $x$ commutes with $A\widetilde{\sigma}\inv=U_1$, so that $x\in \kk^{\varphi}$ if and only if $x\in \kk^{\psi}$. If moreover $x\in \kk^{\varphi}$ (or equivalently, $x\in \kk^{\psi}$), then for any $j\in J$ and $\la\in i\RR$ such that $xe_j^+=\la e_j^+$, we have
$$xe_j^-=x\widetilde{\sigma} e_j^+=\widetilde{\sigma} x e_j^+=-\la \widetilde{\sigma} e_j^+=-\la e_j^-.$$
Since in addition
$$xe_{j_0}=x\widetilde{\sigma} e_{j_0}=\widetilde{\sigma} x e_{j_0}=-xe_{j_0},$$
so that $xe_{j_0}=0$, we deduce that $x\in\ttt_0$. This shows that $\ttt^{\varphi}=\ttt^{\psi}=\ttt_0$ and that $\ttt_0$ is maximal abelian in both $\kk^{\varphi}$ and $\kk^{\psi}$.
The set of roots of $\kk^{\psi}_{\CC}$ with respect to $\ttt^{\psi}_{\CC}$ is then of type $B_{J}$ if $S_1\neq\varnothing$ and of type $C_J$ if $S_1=\varnothing$. Accordingly, the set of compact roots for $\widehat{\LLL}_{\psi}(\kk)$ is of type $BC_J^{(2)}$ or $C_J^{(2)}$ (see \cite[Examples~2.9 and 2.10]{Hloopgroups}).

Let now $\mu\in i\ttt_0^*$ be defined by $\mu(E_{j}):=\mu_j:=-\tfrac{n_j}{2N}$ for all $j\in J$. Setting $\Lambda_{\mu}:=i\mu^{\sharp}=\sum_{j\in J}{i\mu_jE_{j}}\in\uu(\HH)$, we then have
$$\frac{d}{dt}U_{t/2\pi}=-\Lambda_{\mu}U_{t/2\pi}=-U_{t/2\pi}\Lambda_{\mu}.$$

We may thus apply Proposition~\ref{prop:untwisting_extension2} and conclude that for any $\nu\in i\ttt_0^*$, there is an isomorphism
$$\widehat{\Phi}\co \widehat{\LLL}_{\varphi}^{\nu}(\kk)\to \widehat{\LLL}_{\psi}^{\mu+\nu}(\kk):(z_1,\xi(t),z_2)\mapsto (z_1,\phi_{t/2\pi}(\xi(t)),z_2)$$
fixing the Cartan subalgebra $\ttt_0^e:=\RR\oplus\ttt_0\oplus\RR$ pointwise. Moreover, $\widehat{\Phi}$ induces an isomorphism of locally affine root systems $$\pi\co \Delta(\widehat{\LLL}_{\varphi}^{\nu}(\kk),\ttt_0^e)_c\to \Delta(\widehat{\LLL}_{\psi}^{\mu+\nu}(\kk),\ttt_0^e)_c\in\{BC_J^{(2)},C_J^{(2)}\}:(\alpha,n)\mapsto \big(\alpha,\tfrac{2n}{N_{\varphi}}-2\mu(\alpha^{\sharp})\big),$$
where $N_{\varphi}$ is the order of $\varphi$. Finally, the Weyl groups $\WW(\widehat{\LLL}_{\varphi}^{\nu}(\kk),\ttt_0^e)$ and $\WW(\widehat{\LLL}_{\psi}^{\mu+\nu}(\kk),\ttt_0^e)$ coincide.
\end{example}

\medskip
\noindent
{\bf Proof of Theorem~\ref{thm:mainintro}.}
This sums up the results of Examples~\ref{ex:CL}, \ref{ex:HL}, \ref{ex:RL} and \ref{ex:CA}. \hspace{\fill}\qedsymbol

\section{Isomorphisms of Weyl groups}\label{section:COWG}
By Theorem~\ref{thm:mainintro}, every affinisation of a simple Hilbert--Lie algebra $\kk$ is isomorphic to some slanted standard affinisation $\widehat{\LLL}_{\psi}^{\nu}(\kk)$ of $\kk$, 
with the same Weyl groups. In turn, $\widehat{\LLL}_{\psi}^{\nu}(\kk)$ and $\widehat{\LLL}_{\psi}(\kk)$ have the same root system with respect to some common Cartan subalgebra, and hence isomorphic Weyl groups. In this section, we give an explicit isomorphism between these Weyl groups. We then present an application of our results to the study of positive energy highest weight representations of $\widehat{\LLL}_{\psi}^{\nu}(\kk)$.

\subsection{Unslanting Weyl groups}\label{subsection:UWG}
Let $\kk$ be a simple Hilbert--Lie algebra, and let $\psi\in\Aut(\kk)$ be of finite order $N\in\NN$. 
Let $\ttt_0$ be a maximal abelian subalgebra of $\kk^{\psi}$ and let $\Delta_{\psi}=\Delta(\kk,\ttt_0)$ be the corresponding root system. Let also $\ttt_0^e:=\RR\oplus \ttt_0\oplus\RR$ be the corresponding Cartan subalgebra of $\g_{\nu}:=\widehat{\LLL}^{\nu}_{\psi}(\kk)$, for any $\nu\in i\ttt_0^*$. 

We assume that $\psi$ is either the identity or one of the three standard automorphisms (cf. \S\ref{subsection:LARS}), so that $\widehat{\Delta}_{\psi}:=\Delta(\g_{\nu},\ttt_0^e)\subseteq \Delta_{\psi}\times\ZZ$ is one of the $7$ locally affine root systems $A_J^{(1)}$, $B_J^{(1)}$, $C_J^{(1)}$, $D_J^{(1)}$, $B_J^{(2)}$, $C_J^{(2)}$ and $BC_J^{(2)}$. For each $\nu\in i\ttt_0^*$, let $\widehat{\WW}_{\nu}:=\WW(\g_{\nu},\ttt_0^e)\subseteq\GL(\ttt_0^e)$ be the Weyl group of $\g_{\nu}$ with respect to $\ttt_0^e$. 

We fix some $\nu\in i\ttt_0^*$. In the notation of \S\ref{subsection:WG}, we then have a semi-direct decomposition $$\widehat{\WW}_{\nu}=\tau(\TTT_{\psi})\rtimes \WW_{\nu},$$
where $\TTT_{\psi}$ is the abelian subgroup of $\ttt_0$ generated by $\{in\check{\alpha}/N \ | \ (\alpha,n)\in\widehat{\Delta}_{\psi}\}$ and $\WW_{\nu}$ is the subgroup of $\widehat{\WW}_{\nu}$ generated by the reflection $r_{(\alpha,0)}$, $\alpha\in\Delta_{\psi}$, defined by
\begin{equation}\label{eqn:ralpha0bis}
r_{(\alpha,0)}(z,h,t)=(z,h,t)-\big(\alpha(h)+it\nu(\alpha^{\sharp})\big)\cdot \big(-i\nu(\check{\alpha}),\check{\alpha},0\big)\quad\textrm{for all $(z,h,t)\in\ttt_0^e$.}
\end{equation}
We also denote by $\overline{r}_{\alpha}$, $\alpha\in\Delta_{\psi}$, the reflections
$$\overline{r}_{\alpha}(z,h,t)=\big(z,h-\alpha(h)\check{\alpha},t\big)$$
in $\GL(\ttt_0^e)$ generating $\WW_0$, so that $\widehat{\WW}_{0}=\tau(\TTT_{\psi})\rtimes \WW_{0}$.

Finally, we recall from \S\ref{subsection:FOA} that for any $\nu\in i\ttt_0^*$, there is a unique element $\nu^{\sharp}\in i\widehat{\ttt_0}$ such that $\langle \nu^{\sharp},h\rangle=\nu(h)$ for all $h\in i\ttt_0$. We set
$$\widehat{\ttt_0^e}:=\RR\oplus \widehat{\ttt_0} \oplus\RR,$$ and we extend each root $\alpha\in\Delta_{\psi}\subseteq i\ttt_0^*$ to a linear functional 
$$\alpha\co \widehat{\ttt_0}\to i\RR: h\mapsto \alpha(h):= \langle h,\alpha^{\sharp}\rangle,$$
so that $\widehat{\WW}_{0}$ and $\widehat{\WW}_{\nu}$ may be viewed as subgroups of $\GL(\widehat{\ttt_0^e})$.

\begin{lemma}\label{lemma:description_action_W}
Let $\alpha_1,\dots,\alpha_n\in\Delta_{\psi}$ for some $n\in\NN$. Then for all $(z,h,t)\in \widehat{\ttt_0^e}$, 
$$r_{(\alpha_n,0)}\dots r_{(\alpha_2,0)}r_{(\alpha_1,0)}.(z,h,t)-(z,h,t)=\Big(i\nu\big(f_{\nu}^{\alpha_1,\dots,\alpha_n}(z,h,t)\big), -f_{\nu}^{\alpha_1,\dots,\alpha_n}(z,h,t),0\Big),$$
where
$$f_{\nu}^{\alpha_1,\dots,\alpha_n}\co \widehat{\ttt_0^e}\to \ttt_0: (z,h,t)\mapsto \sum_{s=1}^{n}{(\alpha_s(h)+it\nu(\alpha_s^{\sharp}))\cdot r_{\alpha_n}\dots r_{\alpha_{s+1}}(\check{\alpha_s})}.$$
\end{lemma}
\begin{proof}
This easily follows by induction on $n$ using (\ref{eqn:ralpha0bis}) and the decomposition
\begin{equation*}
\begin{aligned}
r_{(\alpha_n,0)}\dots r_{(\alpha_1,0)}.(z,h,t)-(z,h,t) = r_{(\alpha_n,0)}.\big(r_{(\alpha_{n-1},0)}\dots r_{(\alpha_1,0)}.(z,h,t)-(z,h,t)\big)\\
+\big(r_{(\alpha_n,0)}.(z,h,t)-(z,h,t)\big).
\end{aligned}
\end{equation*}
\end{proof}

Since, in the notation of Lemma~\ref{lemma:description_action_W}, 
$$\alpha_s(h)+it\nu(\alpha_s^{\sharp})=\alpha_s(h)+it\langle \nu^{\sharp},\alpha_s^{\sharp}\rangle=\alpha_s(h)+\alpha_s(it\nu^{\sharp})=\alpha_s(h+it\nu^{\sharp}),$$
so that
\begin{equation}\label{eqn:fmu0}
f_{\nu}^{\alpha_1,\dots,\alpha_n}(z,h,t)=f_{0}^{\alpha_1,\dots,\alpha_n}(z,h+it\nu^{\sharp},t)\quad\textrm{for all $(z,h,t)\in \widehat{\ttt_0^e}$,}
\end{equation}
we deduce from Lemma~\ref{lemma:description_action_W} that for all $\alpha_1,\dots,\alpha_n\in \Delta_{\psi}$ and all $(z,h,t)\in \widehat{\ttt_0^e}$,
\begin{equation*}
r_{(\alpha_n,0)}\dots r_{(\alpha_1,0)}.(z,h,t)-(z,h,t)=0 \iff \overline{r}_{\alpha_n}\dots \overline{r}_{\alpha_1}.(z,h+it\nu^{\sharp},t)-(z,h+it\nu^{\sharp},t)=0.
\end{equation*}
This implies in particular that the assignment $r_{(\alpha,0)}\mapsto \overline{r}_{\alpha}$ for each $\alpha\in\Delta_{\psi}$ defines a group isomorphism
$$\gamma_{\nu}\co \WW_{\nu}\to \WW_0:r_{(\alpha_n,0)}\dots r_{(\alpha_1,0)}\mapsto \overline{r}_{\alpha_n}\dots \overline{r}_{\alpha_1}.$$
\begin{prop}\label{prop:isomWeyl}
The map $\gamma_{\nu}\co \WW_{\nu}\to \WW_0$ extends to a group isomorphism
$$\widehat{\gamma}_{\nu}\co \widehat{\WW}_{\nu}\to \widehat{\WW}_{0}: \tau_xw\mapsto \tau_x\gamma_{\nu}(w)\quad (\textrm{$x\in \TTT_{\psi}$, $w\in \WW_\nu$}).$$
\end{prop}
\begin{proof}
This follows from the observation that for each $\alpha\in \Delta_{\psi}$ and $x\in \TTT_{\psi}$ we have
\[r_{(\alpha,0)}\tau_xr_{(\alpha,0)}\inv=\tau_{r_{\alpha}(x)}=\overline{r}_{\alpha}\tau_x\overline{r}_{\alpha}\inv \quad \textrm{where} \quad 
r_{\alpha}\co \ttt_0\to \ttt_0: h\mapsto h-\alpha(h)\check{\alpha}.\qedhere\] 
\end{proof}

\begin{remark}
If $\nu^{\sharp}\in i\ttt_0$, then $\widehat{\gamma}_{\nu}$ is just the conjugation by $\tau_{-i\nu^{\sharp}}$.
\end{remark}

\subsection{Application to positive energy representations}\label{subsection:ATPER}
We place ourselves in the context of \S\ref{subsection:UWG} (although we will only need to assume that $\psi$ is the identity or standard in Proposition~\ref{prop:slantedPEC} and Theorem~\ref{thm:PEC} below) and keep the same notation. 
Set
$$\bc:=(i,0,0)\in i\ttt_0^e\quad\textrm{and}\quad \bd:=(0,0,-i)\in i\ttt_0^e,$$
so that $\bc$ is central in $(\g_{\nu})_{\CC}$ and $\ad(\bd)$ has eigenvalue $\tfrac{n}{N}+\nu(\alpha^{\sharp})$ on $e_n\otimes \kk_{\CC}^{(\alpha,n)}$, $\alpha\in\Delta_{\psi}$ (cf. (\ref{eqn:Dmu})).

For a weight $$\la\in i(\ttt_0^e)^*\cong i\RR\oplus i\ttt_0^*\oplus i\RR,$$
we set $$\la^0:=\la|_{(\ttt_0)_{\CC}}\in i\ttt_0^*, \quad\la_c:=\la(\bc)\in\RR\quad\textrm{and}\quad \la_d:=\la(\bd)\in\RR.$$
Note that we may extend the map $\sharp\co i\ttt_0^*\to i\widehat{\ttt_0}$ to a map $$\sharp\co i(\ttt_0^e)^*\to i\widehat{\ttt_0^e}:\la\mapsto \la^{\sharp}:=-\la_d\bc+(\la^0)^{\sharp}-\la_c\bd$$
satisfying $\kappa((z,h,t),\la^{\sharp})=\la((z,h,t))$ for all $(z,h,t)\in i\ttt_0^e$, where we have again denoted by $\kappa$ the hermitian extension of $\kappa|_{\ttt_0^e\times\ttt_0^e}$ (cf. \S\ref{subsection:DE}) to $(\ttt_0^e)_{\CC}\times (\widehat{\ttt_0^e})_{\CC}$. Since $\sharp$ is bijective, this allows in particular to view the Weyl group $\widehat{\WW}_{\nu}$
not only as a subgroup of $\GL(i\widehat{\ttt_0^e})$ as in \S\ref{subsection:UWG}, but also as a subgroup of $\GL(i(\ttt_0^e)^*)$, where the action is characterised by
$(\widehat{w}.\la)^{\sharp}=\widehat{w}.\la^{\sharp}$ for all $\widehat{w}\in \widehat{\WW}_{\nu}$.

Fix some $\la\in i(\ttt_0^e)^*$, and assume that $\la_c\neq 0$ and that $\la$ is \emph{integral} for $\g_{\nu}$, in the sense that $\la((\alpha,n)^{\vee})\in\ZZ$ for all compact roots $(\alpha,n)\in\widehat{\Delta}_{\psi}$. 
It then follows from \cite[Theorem~4.10]{Ne09} that $\g_{\nu}$ admits an (irreducible) integrable highest-weight module $L_{\nu}(\la)$ of highest weight $\la$, whose corresponding set of weights is given by
\begin{equation}\label{eqn:Pmu}
\PP_{\la}=\PP_{\la}^{\nu}=\conv\!\big(\widehat{\WW}_{\nu}.\la\big)\cap\big(\la+\ZZ[\widehat{\Delta}_{\psi}]\big).
\end{equation}
Let $$\rho_{\la}=\rho_{\la}^{\nu}\co  \g_{\nu}\to\End(L_{\nu}(\la))$$
denote the corresponding representation. Note that $\rho_{\la}$ is unitary with respect to 
some inner product on $L_{\nu}(\la)$ which is uniquely determined up to a positive factor (see \cite[Theorem~4.11]{Ne09}).

Let $\nu'\in i\ttt_0^*$, and extend the derivation $D_{\nu'}=D_0+\overline{D}_{\nu'}$ of $\LLL_{\psi}(\kk)\subseteq\g_{\nu}$ to a skew-symmetric derivation of $\g_{\nu}$ by requiring that $D_{\nu'}(\ttt_0^e)=\{0\}$ (cf. \S\ref{subsection:D}). Since the derivation $D_{\nu'}$  preserves the root space decomposition of $(\g_{\nu})_{\CC}$, it follows from (\ref{eqn:Dmu}) that $\rho_{\la}$ can be extended to a representation
$$\widetilde{\rho}_{\la}=\widetilde{\rho}_{\la}^{\thinspace\nu,\nu'}\co \g_{\nu}\rtimes \RR D_{\nu'}\to \End(L_{\nu}(\la))$$
of the semi-direct product $\g_{\nu}\rtimes \RR D_{\nu'}$ by setting
$$\widetilde{\rho}_{\la}(D_{\nu'})v_{\gamma}:=i\chi\big(\gamma-\la\big)v_{\gamma}$$
for all $\gamma\in\PP_{\la}$ and all $v_{\gamma}\in L_{\nu}(\la)$ of weight $\gamma$, where $\chi\co \ZZ[\widehat{\Delta}_{\psi}]\to \RR$ is the character defined by
\begin{equation}\label{eqn:chi}
\chi((\alpha,n))=\tfrac{n}{N}+\nu'(\alpha^{\sharp})\quad\textrm{for all $(\alpha,n)\in \widehat{\Delta}_{\psi}$}.
\end{equation}
The representation $\widetilde{\rho}_{\la}$ is said to be of \emph{positive energy} if the spectrum of $H_{\nu'}:=-i\widetilde{\rho}_{\la}(D_{\nu'})$ is bounded from below. If this is the case, the infimum of the spectrum of $H_{\nu'}$ is called the \emph{minimal energy level} of $\widetilde{\rho}_{\la}$.
In view of (\ref{eqn:Pmu}), the representation $\widetilde{\rho}_{\la}$ is of positive energy if and only if
$$\inf\chi\big(\widehat{\WW}_{\nu}.\la-\la\big)>-\infty.$$

To match the context of \S\ref{subsection:UWG} (namely, $\widehat{\WW}_{\nu}\subseteq\GL(i\widehat{\ttt_0^e})$ instead of $\widehat{\WW}_{\nu}\subseteq\GL(i(\ttt_0^e)^*)$), 
we identify $\chi$ with an element of $i\widehat{\ttt_0^e}$ by setting
$$\chi(\mu)=\kappa(\mu^{\sharp},\chi)=\mu(\chi)\quad\textrm{for all $\mu\in \ZZ[\widehat{\Delta}_{\psi}]$},$$
so that $$\chi(\widehat{w}.\la-\la)=\kappa(\widehat{w}.\la^{\sharp}-\la^{\sharp},\chi)=\kappa( \widehat{w}\inv.\chi-\chi,\la^{\sharp})=\la(\widehat{w}\inv.\chi-\chi)\quad\textrm{for all $\widehat{w}\in \widehat{\WW}_{\nu}$.}$$ With this identification, $\widetilde{\rho}_{\la}$ is thus a positive energy representation if and only if
\begin{equation}\label{eqn:PEC}
\inf\la\big(\widehat{\WW}_{\nu}.\chi-\chi\big)>-\infty.
\end{equation}
Write $\chi=\chi_c\bc+\chi^0+\chi_d\bd\in i\widehat{\ttt_0^e}$ for some $\chi_c,\chi_d\in\RR$ and some $\chi^0\in i\widehat{\ttt_0}$. Since for all $(\alpha,n)\in \widehat{\Delta}_{\psi}$, \begin{equation}
\begin{aligned}
\chi((\alpha,n))&=\kappa\big((\alpha,n)^{\sharp},\chi\big)=\kappa\big(\big(-i(\tfrac{n}{N}+\nu(\alpha^{\sharp})),\alpha^{\sharp},0\big),\big(i\chi_c,\chi^0,-i\chi_d\big)\big)\\
&=\langle \alpha^{\sharp},\chi^0\rangle +\chi_d(\tfrac{n}{N}+\nu(\alpha^{\sharp}))=\alpha(\chi^0+\chi_d\nu^{\sharp})+\chi_d\tfrac{n}{N},
\end{aligned}
\end{equation}
we then deduce from (\ref{eqn:chi}) that 
\begin{equation}\label{eqn:descmu}
\chi^0=(\nu')^{\sharp}-\nu^{\sharp}\quad\textrm{and}\quad\chi_d=1.
\end{equation}

Define 
$$\la_{\nu}:=\la-\la_c\nu\in i(\ttt_0^e)^*\quad\textrm{and}\quad\chi_{\nu}:=\chi+\chi_d\nu^{\sharp}\in i\widehat{\ttt_0^e},$$
where we view $\nu\in i\ttt_0^*$ as a weight in $i(\ttt_0^e)^*$ by setting $\nu(\bc)=\nu(\bd):=0$. 

\begin{prop}\label{prop:slantedPEC}
Let $\la\in i(\ttt_0^e)^*$ and $\chi=\chi_c\bc+\chi^0+\chi_d\bd\in i\widehat{\ttt_0^e}$. Then
$$\la(\widehat{\WW}_{\nu}.\chi-\chi)=\la_{\nu}(\widehat{\WW}_{0}.\chi_{\nu}-\chi_{\nu}).$$
\end{prop}
\begin{proof}
Let $\widehat{\gamma}_{\nu}\co \widehat{\WW}_{\nu}\to \widehat{\WW}_{0}$ be the group isomorphism provided by Proposition~\ref{prop:isomWeyl}.
We claim that 
\begin{equation*}
\la(\widehat{w}.\chi-\chi)=\la_{\nu}(\widehat{\gamma}_{\nu}(\widehat{w}).\chi_{\nu}-\chi_{\nu})\quad\textrm{for all $\widehat{w}\in \widehat{\WW}_{\nu}$.}
\end{equation*}
Given a tuple of roots $\boldsymbol\alpha=(\alpha_1,\dots,\alpha_n)\in \Delta_{\psi}^n$, we use the notation
$$r_{(\boldsymbol\alpha,0)}= r_{(\alpha_n,0)}\dots r_{(\alpha_1,0)}\quad\textrm{and}\quad \overline{r}_{\boldsymbol\alpha}=\overline{r}_{\alpha_n}\dots \overline{r}_{\alpha_1}.$$
It then follows from Lemma~\ref{lemma:description_action_W} and $(\ref{eqn:fmu0})$ that for all $(z,h,t)\in i\widehat{\ttt_0^e}$,
\begin{equation*}
\left\{
\begin{array}{ll}
r_{(\boldsymbol\alpha,0)}.(z,h,t)-(z,h,t)=\Big(i\nu\big(f_{\nu}^{\boldsymbol\alpha}(z,h,t)\big), -f_{\nu}^{\boldsymbol\alpha}(z,h,t),0\Big),\\
\overline{r}_{\boldsymbol\alpha}.(z,h,t)-(z,h,t)=-\Big(0, f_{\nu}^{\boldsymbol\alpha}(z,h-it\nu^{\sharp},t),0\Big).
\end{array}
\right.
\end{equation*}
Let now $\widehat{w}\in \widehat{\WW}_{\nu}$, which we write as $\widehat{w}=\tau_xr_{(\boldsymbol\alpha,0)}$ for some $x\in \TTT_{\psi}$ and some tuple of roots $\boldsymbol\alpha$. Then
\begin{equation*}
\begin{aligned}
\la_{\nu}\big(\widehat{\gamma}_{\nu}(\widehat{w}).\chi_{\nu}-\chi_{\nu}\big) &= \la_{\nu}\big(\tau_x\overline{r}_{\boldsymbol\alpha}.(i\chi_c,\chi^0+\chi_d \nu^{\sharp},-i\chi_d)-(i\chi_c,\chi^0+\chi_d \nu^{\sharp},-i\chi_d)\big)\\
&= \la_{\nu}\big(\tau_x.(i\chi_c,\chi^0+\chi_d\nu^{\sharp}-f_{\nu}^{\boldsymbol\alpha}(\chi),-i\chi_d)-(i\chi_c,\chi^0+\chi_d \nu^{\sharp},-i\chi_d)\big)\\
&= \la_{\nu}\big(\big(-\chi_d\langle \nu^{\sharp},x\rangle-\langle\chi^0-f_{\nu}^{\boldsymbol\alpha}(\chi),x\rangle+\tfrac{i\chi_d\langle x,x\rangle}{2},-f_{\nu}^{\boldsymbol\alpha}(\chi)-i\chi_dx,0\big)\big)\\ 
&= i\la_c\big(\langle\chi^0-f_{\nu}^{\boldsymbol\alpha}(\chi),x\rangle-\tfrac{i\chi_d\langle x,x\rangle}{2}\big)-\la^0\big(f_{\nu}^{\boldsymbol\alpha}(\chi)+i\chi_dx\big)+\la_c\nu\big(f_{\nu}^{\boldsymbol\alpha}(\chi)\big)\\
&= \la\big(\big(i\nu\big(f_{\nu}^{\boldsymbol\alpha}(\chi)\big)-\langle\chi^0-f_{\nu}^{\boldsymbol\alpha}(\chi),x\rangle+\tfrac{i\chi_d\langle x,x\rangle}{2}, -f_{\nu}^{\boldsymbol\alpha}(\chi)-i\chi_dx,0\big)\big)\\
&= \la\big(\tau_x.\big(i\chi_c+i\nu\big(f_{\nu}^{\boldsymbol\alpha}(\chi)\big), \chi^0-f_{\nu}^{\boldsymbol\alpha}(\chi),-i\chi_d\big)-(i\chi_c,\chi^0,-i\chi_d)\big)\\
&= \la\big(\tau_xr_{(\boldsymbol\alpha,0)}.(i\chi_c, \chi^0,-i\chi_d)-(i\chi_c,\chi^0,-i\chi_d)\big)\\
&= \la\big(\widehat{w}.\chi-\chi\big).
\end{aligned}
\end{equation*}
This concludes the proof of the proposition.
\end{proof}

\begin{theorem}\label{thm:PEC}
Let $\la\in i(\ttt_0^e)^*$ be an integral weight for $\g_{\nu}$ with $\la_c=\la(\bc)\neq 0$. Then for any $\nu,\nu'\in i\ttt_0^*$, the following assertions are equivalent:
\begin{enumerate}
\item[(i)]
The highest weight representation $\widetilde{\rho}_{\la}^{\thinspace\nu,\nu'}\co \g_{\nu}\rtimes \RR D_{\nu'}\to \End(L_{\nu}(\la))$ is of positive energy.
\item[(ii)]
$M_{\nu,\nu'}:=\inf \la_{\nu}(\widehat{\WW}_{0}.\chi_{\nu}-\chi_{\nu})>-\infty$, where $\la_{\nu}:=\la-\la_c\nu\in i(\ttt_0^e)^*$ and $\chi_{\nu}:=(\nu')^{\sharp}+\bd\in i\widehat{\ttt_0^e}$.
\end{enumerate}
Moreover, if $M_{\nu,\nu'}>-\infty$, then $M_{\nu,\nu'}$ is the minimal energy level of $\widetilde{\rho}_{\la}^{\thinspace\nu,\nu'}$.
\end{theorem}
\begin{proof}
This readily follows from (\ref{eqn:PEC}) and Proposition~\ref{prop:slantedPEC}, where the above description of $\chi_{\nu}$ follows from (\ref{eqn:descmu}).
\end{proof}

\medskip
\noindent
{\bf Proof of Theorem~\ref{thm:PECintro}.}
By Theorem~\ref{thm:mainintro}(iii), the Weyl groups of $\widehat{\LLL}_{\varphi}^{\nu}(\kk)$ and $\widehat{\LLL}_{\psi}^{\mu+\nu}(\kk)$ with respect to $\ttt_0^e$ coincide. By (\ref{eqn:PEC}), the representation $\widetilde{\rho}_{\la}\co \widehat{\LLL}_{\varphi}^{\nu}(\kk)\rtimes \RR D_{\nu'}\to \End(L_{\nu}(\la))$ is thus of positive energy if and only if 
$$\inf\la\big(\widehat{\WW}_{\mu+\nu}.\chi-\chi\big)>-\infty,$$
where $\widehat{\WW}_{\mu+\nu}=\widehat{\WW}_{\psi}$ is the Weyl group of the slanted standard affinisation $\widehat{\LLL}_{\psi}^{\mu+\nu}(\kk)$ of $\kk$ and $\chi=(\nu')^{\sharp}-\nu^{\sharp}+\bd$ (see (\ref{eqn:descmu})). Thus Theorem~\ref{thm:PECintro} follows from Proposition~\ref{prop:slantedPEC}, where the character $\chi$ in the statement of Theorem~\ref{thm:PECintro} corresponds to $\chi_{\mu+\nu}=(\nu')^{\sharp}+\mu^{\sharp}+\bd$ in the above notation.
 \hspace{\fill}\qedsymbol

\bibliographystyle{amsalpha} 
\bibliography{these}

\end{document}